\documentclass[a4paper,10pt]{article}
\usepackage{cite}
\usepackage{latexsym}
\usepackage{amsmath,amsthm}
\usepackage{enumerate}
\usepackage{amssymb}
\usepackage{upgreek}
\usepackage{authblk}

\usepackage[margin=1.25in,dvips]{geometry}

\newcommand{\vol}{\textnormal{vol}}

\def\f12{\frac 1 2}

\def\a{\alpha}
\def\b{\beta}
\def\ga{\gamma}

\def\ep{\epsilon}

\def\si{\sigma}
\def\Si{\Sigma}
\def\om{\omega}
\def\Om{\Omega}

\def\pa{\partial}
\def\les{\lesssim}

\def\f12{\frac 1 2}

\newcommand{\lap}{\mbox{$\Delta \mkern-13mu /$\,}}

\newcommand{\nabb}{\mbox{$\nabla \mkern-13mu /$\,}}
\newtheorem{thm}{Theorem}

\newtheorem{prop}{Proposition}

\newtheorem{lem}{Lemma}
\newtheorem{cor}{Corollary}
\newtheorem{remark}{Remark}
\begin{document}

\title{Global solutions of nonlinear wave equations in time dependent inhomogeneous media}
\author{Shiwu Yang}
\date{}
\maketitle

\begin{abstract}
We consider the problem of small data global existence for a class of semilinear wave equations with null condition on a Lorentzian background $(\mathbb{R}^{3+1}, g)$ with a \textbf{time dependent metric $g$} coinciding with Minkowski metric outside the cylinder $\{\left.(t, x)\right| |x|\leq R\}$. We show that the small data global existence result can be reduced to two integrated local energy estimates and demonstrate these estimates in the particular case when $g$ is merely $C^1$ close to the Minkowski metric. One of the novel aspects of this work is that it applies to equations on backgrounds which do not settle to any particular stationary metric.
\end{abstract}

\section{Introduction}

In this paper, we study the behavior of solutions to the Cauchy problem
\begin{equation}
\label{THEWAVEEQ}
\begin{cases}
 \Box_g\phi=\frac{1}{\sqrt{-G}}\pa_\a\left(g^{\a\b}\sqrt{-G}\pa_{\b}\phi\right)=F(\phi, \pa\phi),\quad G=\text{det}(g),\\
 \phi(0,x)=\ep \phi_0(x), \pa_t\phi(0,x)=\ep \phi_1(x)
\end{cases}
\end{equation}
on a Lorentzian manifold $(\mathbb {R}^{3+1}, g)$ with initial data $\phi_0(x)$, $\phi_1(x)\in C_0^{\infty}(\mathbb{R}^n)$. The nonlinearity $F$ is assumed to satisfy the null condition with respect to the Minkowski metric
$$F(0, 0)=0, \quad dF(0, 0)=0,
$$
\begin{equation}
 \label{nullcond}
F(\phi,
\pa\phi)=A^{\a\b}\pa_\a\phi\pa_\b\phi+O(|\phi|^3+|\pa\phi|^3),
\end{equation}
in which $A^{\a\b}$ are constants and $A^{\a\b}\xi_\a \xi_\b=0$ whenever
$\xi_0^2=\xi_1^2+\xi_2^2+\xi_3^2$.

\bigskip

The corresponding Cauchy problem in Minkowski space
\begin{equation}
\label{THEWAVEEQM}
\begin{cases}
 \Box\phi=F(\phi, \pa\phi, \pa^2\phi), \\
 \phi(0,x)=\ep \phi_0(x), \pa_t\phi(0,x)=\ep \phi_1(x)
\end{cases}
\end{equation}
has been studied extensively. A classical result is the short time existence for large initial data \cite{hormander}, ~\cite{sogge}. The problem of long time behavior of the solutions with small data has been investigated by Fritz John ~\cite{john-lspan}, ~\cite{john-lowerb83}. In ~\cite{john-lspan}, application of the standard energy methods led to a lower bound on the time existence $T\geq O(\frac{1}{\ep})$, while in ~\cite{johnblowup}, it was shown that any nontrivial $C^3$ solution of the equation
$$\Box \phi=\phi_t^2
$$
with compactly supported initial data blows up in finite time. For
general nonlinearity, an almost global existence with a lower bound
on the time existence $T\geq C \exp\{\frac{C}{\ep}\}$ was
established in ~\cite{kl-johnalmostge}. These results and the
corresponding results in higher dimensions ~\cite{klgex},
~\cite{kl-ponce}, ~\cite{ge-shatah} rely on the decay properties of
solutions to a linear wave equation on $\mathbb{R}^{n+1}$.

\bigskip

A remarkable progress in understanding the problem of small data
global existence in three and higher dimensions has been achieved by
S. Klainerman ~\cite{klinvar}. His approach relied on the vector
field method, which connects the symmetries of the flat
$\mathbb{R}^{n+1}$ with the quantitative decay properties of
solutions of a linear wave equation and led to a small data global
existence in dimension $n>3$ and an almost global existence in
$\mathbb{R}^{3+1}$. Furthermore, in $\mathbb{R}^{3+1}$, a sufficient
condition on a quadratic nonlinearity $F$, which guarantees the
small data global existence, is the celebrated null condition
introduced by S. Klainerman ~\cite{klNullc}. Under this condition,
D. Christodoulou ~\cite{ChDNull} and S. Klainerman ~\cite{klNull}
independently proved the small data global existence result.

\bigskip

The approach of ~\cite{ChDNull} used the conformal method, which relies on the embedding of Minkowski space to the Einstein cylinder $R\times S^{3}$. In ~\cite{klNull}, S. Klainerman used the vector field method based on the weighted energy inequalities generated by the vector fields
\begin{equation}
\label{Lorenzinv} \Gamma=\left\{ \Om_{ij}=x_i\pa_j-x_j\pa_i,
L_i=t\pa_i+x_i\pa_t, \pa_\a, K=(t^2+r^2)\pa_t + 2tr\pa_r,
S=t\pa_t+r\pa_r\right\}.
\end{equation}

The vector field method found many applications, in particular to systems of
nonlinear wave equations on $\mathbb{R}^{3+1}$ with multiple speeds
  ~\cite{klmulti}, ~\cite{sideris-multispeed}, ~\cite{soggemulti} and obstacle
   problems ~\cite{sogge-metcalfe-nakamura}, ~\cite{sogge-metcalfe}. Due to
    the lack of Lorentz invariance, these works avoided the use of the hyperbolic
     rotations $L_i$, but still used the scaling vector field $S$.

\bigskip

Another application of the vector field method is to the quasilinear wave equation of the form
\begin{equation}
\label{Twaveeq}
\Box_{g(\phi)}\phi=F(\phi, \pa\phi)
\end{equation}
with $g(0)=m$, the Minkowski metric.  The quasilinear part of the equation
 $g^{\a\b}(\phi)\pa_{\a\b}\phi$ never satisfies the null condition. Nevertheless, 
several examples of ~\eqref{Twaveeq} still admit global solutions.
  In ~\cite{gx-lindblad}, ~\cite{gx-lindblad2}, H. Lindblad
  obtained the small data global existence result of the equation
\begin{equation*}
 g^{\a\b}(\phi)\pa_{\a\b}\phi=0
\end{equation*}
on $\mathbb{R}^{3+1}$. A particular case
\[
 \pa_{tt}\phi-(1+\phi)^2\Delta \phi=0
\]
was investigated previously by S. Alinhac ~\cite{alinhac-example}.

\bigskip

Separately, the motivation for studying ~\eqref{Twaveeq} arises
 from the problem of global nonlinear stability of Minkowski
  space originally proved by Christodoulou-Klainerman by recasting
   the problem as a system of Bianchi equations for the curvature
    tensor ~\cite{kcg}. Later, Lindblad-Rodnianski ~\cite{SMigor}
     obtained a different proof of stability of Minkowski in wave coordinates,
      in which the problem was formulated as a system of quasilinear wave equations for the components of the metric.

\bigskip

We should also mention  that a linear problem
\[
 \Box_{g(t, x)} \phi=0
\]
has been studied in ~\cite{alinhac-freedecay}. There it was shown that  $\phi$ has the decay properties similar to those of a solution of a linear wave equation on Minkowski space provided that $g(t, x)$ approaches the Minkowski metric $m$ suitably as $t\rightarrow \infty$ .

\bigskip
 A common feature of these problems is that the background metric $g$ converges to the Minkowski metric either by the assumptions in ~\cite{alinhac-freedecay} or, for the equation ~\eqref{Twaveeq}, by the assumption $g(0)=m$ and the expected convergence $\phi(t, x)\rightarrow 0$ as $t\rightarrow\infty$. The need for such convergence, or at least convergence to some stationary metric $g(x)$ is dictated by the vector field method. All applications of the vector field method require commutations with generators of the symmetries of Minkowski space, at the very least with the scaling vector field $S=t\pa_t +r\pa_r$. For the problem
\[
 \Box_g \phi=F,
\]
the error term coming from the commutation with $S$ (or $L_i$) would be of the form
$t \pa_t g^{\a\b}\pa_{\a\b}\phi$ which leads to the requirement that $t\pa_t g$ is at least bounded and thus $g$ decays to a stationary metric.

\bigskip

Therefore, the vector field method has not been so far successful in application to the equation
\begin{equation}
 \label{THEWAVEEQMTH}
\begin{cases}
 \Box_{g(t, x)}\phi=F(\phi, \pa\phi) ,\\
 \phi(0,x)=\ep \phi_0(x), \pa_t\phi(0,x)=\ep \phi_1(x),
\end{cases}
\end{equation}
in which $g(t, x)$ does not converge to a stationary metric $g(x)$.
These problems describe propagation of nonlinear waves in a
time-dependent inhomogeneous medium.

\bigskip

In this paper, we develop a new approach to treat the small data global existence problem for the equation ~\eqref{THEWAVEEQMTH}. This approach relies on a new method for proving decay for linear problems, developed by M. Dafermos and I. Rodnianski in ~\cite{newapp}. We now describe the assumptions and the main results.
\bigskip

We assume that the metric $g$ coincides with the Minkowski metric
outside a cylinder
$$g_{\a\b}=m_{\a\b}+ h_{\a\b}, \quad g^{\a\b}=m^{\a\b} + h^{\a\b},$$
where $h_{\a\b}$ are smooth functions supported in $\{(t, x)||x|\leq\f12 R\}$, $R$ is a fixed constant
 and $m$ is the Minkowski  metric. We make a convention that the Greek indices run from 0 to 3 while the Latin
 indices run from 1 to 3. 

In coordinates $(t, x)$, we denote $\pa_0=\pa_t$, $\pa_i=\pa_{x_i}$, $\pa=(\pa_t, \pa_1, \pa_2, \pa_3)$.
We also use the standard polar coordinates in Minkowski space $(t, r, \om)$. Let $\nabb$ denote the induced covariant
 derivative and $\lap$ the induced Laplacian on the spheres of constant $r$. And $\Om$ is a shorthand for
 the angular momentum $\Om_{ij}=x_i\pa_j-x_j\pa_i$. We also define the null coordinates $u=\frac{t-r}{2}$, $v=\frac{t+r}{2}$. And denote $T$ as the vector field $\pa_t$ in $(t, r,\om)$ coordinates.

\bigskip

In our argument, we estimate the decay of the solution with respect
to the foliation $\Si_{\tau}$, defined as follows:
\begin{align*}
&S_\tau:=\{u=u_\tau, v\geq v_\tau\},\\
&\Si_\tau:=\{t=\tau, r\leq R\}\cup S_\tau,
\end{align*}
where $u_\tau=\frac{\tau-R}{2}$, $v_\tau=\frac{\tau+R}{2}$. Thus the
corresponding energy flux in Minkowski space is
$$ E[\phi](\tau):=\int_{r\leq R}|\pa_t\phi|^2+|\pa_r\phi|^2+|\nabb\phi|^2dx +
\int_{S_\tau}\left(|\pa_v\phi|^2+|\nabb\phi|^2\right)r^2dvd\om.
$$

Additional to the assumption that $g$ coincides with the Minkowski metric outside the cylinder $\{(t, x)||x|\leq\f12 R\}$, we make two other assumptions:
\begin{itemize}
\item[$\mathcal A1$]: \textsl{There exists a positive constant $\lambda$ such that in $(t, x)$ coordinates
$$g_{00}\leq - \lambda, \quad \lambda|x|^2\leq g_{ij}x_i x_j\leq \lambda^{-1}|x|^2
$$
for any $x=(x_1, x_2, x_3)\in \mathbb {R}^3$.
\item[$\mathcal A2$]: There exists a positive number $\a<1$ and a constant $C_0$ such that two integrated local energy inequalities hold for any smooth function} $\phi(t, x)$
\begin{align}
\label{morawetz1}
&\int_{\tau_1}^{\tau_2}\int_{r\leq R}|\pa\phi|^2+\frac{\phi^2}{r}dxdt\leq C_0 E[\phi](\tau_1) + C_0
D[\Box_g\phi]_{\tau_1}^{\tau_2},\\
\label{morawetz2} &\int_{\tau_1}^{\tau_2}\int_{r\leq \f12
R}|\pa\phi_t|^2dxdt\leq C_0\left( E[\phi_t](\tau_1)
+D[\pa_t\Box_g\phi]_{\tau_1}^{\tau_2}+E[\phi](\tau_1)
+D[\Box_g\phi]_{\tau_1}^{\tau_2}\right),
\end{align}
\end{itemize}
where we denote
$$D[F]_{\tau_1}^{\tau_2}:=\int_{\tau_1}^{\tau_2}\int_{\Si_\tau}|F|^2(1+r)^{\a+1}dxd\tau.
$$
We define two sets
\begin{align*}
&A:=\{(k, j)|k+j\leq 8, k\leq 5\},\\
&B:=\{(k, j)|(k, j+2)\in A\},
\end{align*}
where $k$, $j$ are always nonnegative integers. And then denote
$$E_0=\sum\limits_{(k, j)\in A}E[\Om^k T^j\phi](0),
$$
which is determined by the initial data $\phi_0(x)$, $\phi_1(x)$ and the equation ~\eqref{THEWAVEEQMTH}.

\bigskip

Our main results are:

\begin{thm}
\label{maintheorem} Suppose the nonlinearity $F$ satisfies the null
condition ~\eqref{nullcond} and $g$ satisfies $\mathcal A1$ and
$\mathcal A2$. Assume that the initial data $\phi_0(x)$, $\phi_1(x)$
are smooth and supported in $\{|x|\leq R\}$. Then there exists
$\ep_0>0$, depending on $R$, $\a$, $E_0$, $\lambda$, $C_0$, $h$,
such that for all positive $\ep<\ep_0$, the equation
~\eqref{THEWAVEEQMTH} admits a unique global smooth solution.
Moreover, for the solution $\phi$, we have
\begin{itemize}
\item[(1)] Energy decay
$$E[\phi](\tau)\leq C E_0\ep^2  (1+\tau)^{-2+\a}
$$
for some constant $C$ depending on $R$, $\a$, $\lambda$, $C_0$ and $h$.
\item[(2)] Pointwise decay: for any $\a<\delta\leq 1$
 \begin{eqnarray*}
|\phi|\leq C_{\delta} \sqrt{E_0}\ep (1+r)^{-1}(1+|t-r+R|)^{-\f12+\f12\delta},\\
\sum\limits_{|\b|\leq 2}|\pa^\b\phi|\leq C \sqrt{E_0}\ep
(1+r)^{-\f12}(1+|t-r+R|)^{-1+\f12\a},
\end{eqnarray*}
where $C_{\delta}$ also depends on $\delta$.
\end{itemize}
\end{thm}

\begin{remark}
 The argument here is also applicable to the corresponding problems in higher dimension without null condition.
\end{remark}

\begin{remark}
 The same conclusion holds if the assumption $\mathcal{A}_2$ is replaced with that $g$ satisfies ~\eqref{morawetz1} and the
deformation tensor $\pi^{\pa_t}_{\a\b}=\pa_t g_{\a\b}$ is small, independent of the initial data. We remark
 that this is consistent with our attempt to investigate nonlinear wave equations on backgrounds far from Minkowski space.
\end{remark}

\begin{remark}
 It is not necessary to require that the initial data have compact support. The general assumption on the initial data can be that the following quantity
\[
 \sum\limits_{(k, j)\in A}\iint_{\mathbb{R}^{3}}r^{2-\a}|\pa\Om^k T^j\phi(0, x)|^2dx
\]
is sufficiently small, where $\a$ comes from $\mathcal{A}_2$ or the smallness assumption ~\eqref{smallnesscond}.
\end{remark}

The condition $\mathcal{A}1$ is easy to satisfy. It is equivalent to say that the background is uniformly hyperbolic
and the vector field $\pa_t$ is uniformly timelike. Therefore we reduce the small data global existence to the
two integrated local energy inequalities. Below, we describe particular conditions for which $\mathcal{A}_2$
 can be explicitly verified. We will show that under the assumption that $g$ is $C^1(\mathbb{R}^{3+1})$
 close to the Minkowski metric, $\mathcal{A}_1$ and $\mathcal{A}_2$ hold and hence Theorem ~\ref{maintheorem} follows. More precisely, denoting
$$H=\max \{\|h_{\a\b}\|_{C^1}, \|h^{\a\b}\|_{C^1}\},
$$
then

\begin{thm}
\label{thm2} Suppose $h$ is supported in the cylinder $\{(t,
x)||x|\leq \frac{R}{2}\}$. Then there exists a positive constant
$\ep_0$, depending only on $R$, such that if $H<\ep_0$, then $g$
satisfies
 conditions $\mathcal{A}1$ and $\mathcal{A}2$. Hence Theorem ~\ref{maintheorem} holds for some $\a<1$.
\end{thm}

\begin{remark}
From Theorem \ref{thm2}, we retrieve the classical result proved by
Klainerman \cite{klNull} and Christodoulou ~\cite{ChDNull} in
Minkowski space where $H=0$.
\end{remark}

The main difficulty of considering nonlinear wave equations in inhomogeneous media is the lack of symmetries compared to Minkowski space. In order to make use of those symmetries in Minkowski space, previous works have relied on the fact that the background metric decays to its stationary state, which is not satisfied in this context. In fact, we even allow the background metric $g$ to stay far from the flat one provided that we have two integrated local energy inequalities.

\bigskip

In our approach, we avoid the use of vector fields containing positive weights in $t$,
e.g, $S=t\pa_t+r\pa_r$, $L_i=x_i\pa_t+t\pa_i$. Traditionally, vector fields from the set $\Gamma$
 ~\eqref{Lorenzinv} are used both as multipliers and commutators. In this paper, we only commute
  with $\pa_t$, $\Om_{ij}$. The role of multiplier vector fields is played by $\pa_t$, $f\frac{\pa}{\pa r}$,
   where $f$ is some appropriate function, in the derivation of the integrated local energy decay,
    and by the family of vector fields $r^p(\pa_t+\pa_r)$ \textbf{localized} to a far away region $r\geq R$.
     The key in our argument is a new approach, developed by M. Dafermos and I. Rodnianski in ~\cite{newapp},
     to the problem of decay, in particular, of the energy flux $E[\phi](\tau)$ for solutions of linear wave
      equations. This new approach is a combination of an integrated local energy inequality and a p-weighted
       energy inequality in a neighborhood of the null infinity. We will discuss them in details in Section 2 and
        Section 3.

\bigskip

The plan of this paper is as follows: we first establish an
integrated energy inequality in the whole space time in Section 2 by
using the vector field method; then in Section 3, we revisit the
p-weighted
 energy inequality developed in ~\cite{newapp} and prove the decay of the energy flux $E[\phi](\tau)$.
 In Section 4, we use elliptic estimates to get the pointwise decay of the solution; and then in the last two
  sections, we close our boostrap argument and conclude our main theorems.

\section{Integrated Local Energy Inequality}
In this section,  we use the multiplier method to prove an
integrated energy estimate in the whole space time under the
conditions $\mathcal {A}1$ and $\mathcal{A}2$ or the smallness
assumption on $H$. This kind of estimates were first proven in
~\cite{mora2}. Here we would follow the way in ~\cite{dr3}, also see
~\cite{ILEsterbenz}.

We recall the energy-momentum tensor
\[
{\mathbb T}_{\mu\nu}[\phi]=\pa_\mu\phi\pa_\nu\phi-\frac12 g_{\mu\nu}\pa^{\gamma}\phi\pa_{\gamma}\phi.
\]

Given a vector field $X$, we define the currents
\[
J^X_\mu[\phi]= {\mathbb T}_{\mu\nu}[\phi]X^\nu, \qquad
K^X[\phi]= {\mathbb T}^{\mu\nu}[\phi]\pi^X_{\mu\nu},
\]
where $\pi^X_{\mu\nu}=\frac12 \mathcal{L}_Xg_{\mu\nu}$ is the deformation tensor of the vector field $X$.
Recall that
\[
D^\mu J^X_\mu[\phi] = X(\phi)\Box_g\phi+K^X[\phi].
\]

We denote $n$ as the unit normal vector field to hypersurfaces and $d\si$ the induced measure, $d\vol$ the
volume form of $(R^{3+1}, g)$. Denote the null infinity from $\tau_1$ to $\tau_2$ as
\begin{equation*}
\mathcal I_{\tau_1}^{\tau_2} :=\{(u,v,\omega)|u_{\tau_1}\leq u \leq u_{\tau_2}, v=\infty\}
\end{equation*}
and the corresponding energy flux
$$I[\phi]_{\tau_1}^{\tau_2}:=\left.\int_{\mathcal I_{\tau_1}^{\tau_2}}\left((\pa_u\phi)^2+|\nabb\phi|^2\right)
 r^2dud\om\right|_{v=\infty},
$$
which can be interpreted as a limit when $v\rightarrow\infty$.

Define the modified energy
$$\tilde{E}[\phi](\tau)=E[\phi](\tau)+I[\phi]_{0}^{\tau}.
$$

Take a vector field defined as follows
$$X=f\pa_r=f\frac{x_i}{r}\pa_i,$$ where $f$ is a function of $r$.
Consider the region bounded by the hypersurfaces $\Si_{\tau_1}$ and
$\Si_{\tau_2}$. Using Stoke's formula,  we have
\begin{align}
\notag&\quad\int_{{\Sigma}_{\tau_1}}J^X_\mu[\phi]n^\mu d\sigma - \int_{{\Sigma}_{\tau_2}}J^X_\mu[\phi]n^\mu d\sigma-\int_{\mathcal I_{\tau_1}^{\tau_2}}J^X_\mu[\phi]n^\mu d\sigma\\
&=\int_{\tau_1}^{\tau_2}\int_{\Sigma_\tau}D^\mu J^X_{\mu}[\phi]d\vol=\int_{\tau_1}^{\tau_2}\int_{\Sigma_\tau}FX(\phi)
 + K^X[\phi] d\vol,
\label{energyeq}
\end{align}
where
\begin{align*}
K^X[\phi]=\mathbb{T}^{\mu\nu}[\phi]\pi^X_{\mu\nu}&=\pa_j(f\frac{x_i}{r})\pa^{j}\phi\cdot \pa_i \phi-(\f12 f'+r^{-1}f)\pa^{\gamma}\phi \pa_{\gamma}\phi\\
& \quad+\f12 f\frac{x_i}{r}\pa_i g_{\mu\nu}\cdot \pa^{\mu}\phi
\pa^{\nu}\phi-\frac{1}{4}f\frac{x_i}{r} \pa_{i}g_{\mu\nu}\cdot
g^{\mu\nu}\pa^{\gamma}\phi\pa_{\gamma}\phi,
\end{align*}
in which we denote $f'$ as $\pa_r f$.

Choose another radial symmetric function $\chi$ of $r$. We have the following equality
\begin{align*}
 -\chi\pa^{\gamma}\phi\pa_{\gamma}\phi + \f12\Box_g\chi\cdot\phi^2
 &=-\f12 \chi\left(\Box_g\phi^2 - 2\phi \Box_g\phi\right)+\f12 \Box_g\chi \cdot\phi^2\\
 &= \f12\left(\Box_g\chi\cdot \phi^2 - \chi\Box_g\phi^2\right)+\chi\phi\Box_g\phi\\
 &=\f12 D^{\mu}\left(\pa_{\mu}\chi\cdot \phi^2 - \chi\pa_{\mu}\phi^2\right) +
 \chi\phi\Box_g\phi.
\end{align*}
Add the above equality to both sides of \eqref{energyeq} and modify the current as
\begin{equation}
\label{mcurent} \tilde{J}_{\mu}^X[\phi]=J_{\mu}^X[\phi] -
\f12\pa_{\mu}\chi \cdot\phi^2 + \f12 \chi\pa_{\mu}\phi^2.
\end{equation}
Then we obtain
\begin{align}
\label{menergyeq}
&\int_{{\Sigma}_{\tau_1}}\tilde{J}^X_\mu[\phi]n^\mu d\sigma - \int_{{\Sigma}_{\tau_2}}\tilde{J}^X_\mu[\phi]n^\mu d\sigma-\int_{\mathcal I_{\tau_1}^{\tau_2}}\tilde{J}^X_\mu[\phi]n^\mu d\sigma\\
\notag&=\int_{\tau_1}^{\tau_2}\int_{\Sigma_\tau}
FX(\phi)+F\phi\chi+(\chi-r^{-1}f+\f12 f')(\pa_r\phi)^2 - \f12\Box_g\chi\cdot\phi^2\\
\notag& \quad+(r^{-1}f + \f12 f' - \chi)(\pa_t\phi)^2+(\chi-\f12
f')|\nabb\phi|^2+ error,
\end{align}
where
\begin{align}
\notag
error &= \pa_j(f\frac{x_i}{r})h^{j\mu}\pa_{\mu}\phi\cdot \pa_i \phi-(\f12 f'+r^{-1}f-\chi)(\pa^{\gamma}\phi \pa_{\gamma}\phi+|\pa_t\phi|^2-|\pa_i\phi|^2)\\
 \label{error}
&\quad-\f12 f\frac{x_i}{r}\pa_i g^{\mu\nu}\cdot \pa_{\mu}\phi
\pa_{\nu}\phi-\frac{1}{4}f\frac{x_i}{r}\pa_{i} g_{\mu\nu}\cdot
g^{\mu\nu}\pa^{\gamma}\phi\pa_{\gamma}\phi.
\end{align}
Here recall that $h^{\a\b}=g^{\a\b}-m^{\a\b}$.

\bigskip

The idea is that we choose functions $f$ and $\chi$ such that the
coefficients on the right hand side of \eqref{menergyeq} are
positive. And then control the left hand side by the energy
$\tilde{E}[\phi]$. Thus we end up with an integrated energy
inequality in the whole space time. To proceed, let's first prove
several lemmas in order to estimate the left hand side of
~\eqref{menergyeq}.

We first show that the spherical average of the solution near the null infinity can be bounded by the energy.
\begin{lem}
\label{lem1} Let $\phi(t, x)\in C^{\infty}(\mathbb{R}^{3+1})$. Then
we have
$$r\int_{\om}|\phi(t, r, \om)|^2 d\om\leq \tilde{E}[\phi](\tau),\qquad \forall (t, r, \om)\in S_\tau\cup
\mathcal I_{0}^{\tau}.$$
\end{lem}
\begin{proof} Suppose the point can be represented as $(u,v,\om)=(t, r, \om)$ in our coordinate systems.
It suffices to consider the case when $\tilde{E}(\tau)$ is finite.
We claim that $\phi$ vanishes at the null infinity
$\mathcal{I}_{0}^{u_\tau}$. In fact since $I[\phi]_{0}^{\tau}$ is
finite, we conclude that $
 \int_{\mathcal I_0^{u_\tau}}\phi_u^2d\si
$ is finite. Recall that $r=v-u$. We infer that
\begin{equation*}
 \int_{\om}\left(\int_{0}^{u_\tau}|\pa_u\phi|du\right)^2d\om\leq \int_{\mathcal I_0^{u_\tau}}\phi_u^2d\si\cdot
 \left.\int_{0}^{u_\tau} r^{-2}du\right|_{v=\infty}=0.
\end{equation*}
Hence $\pa_u\phi$ vanishes on $\mathcal{I}_{0}^{u_\tau}$. Notice
that the initial data are supported in $r\leq R$. Then the finite
speed of propagation for wave equations ~\cite{sogge} implies that
$\phi$ vanishes on $S_0$. We thus conclude that $\phi$ vanishes on
$\mathcal{I}_{0}^{u_\tau}$. Therefore
\begin{equation}
\label{bd} \int_{\om}|\phi|^2d\om
=\int_{\om}\left(\int_{v}^{\infty}\pa_v\phi dv\right)^2d\om
\leq\int_{S_\tau}\phi_v ^2d\si\cdot\int_{v}^\infty r^{-2}dv
\leq\frac{1}{r}E[\phi](\tau),
\end{equation}
where on $S_{\tau}$, $d\si=r^2dvd\om$. Hence the lemma holds.
\end{proof}

\bigskip

The following analogue of Hardy's inequality will be used later on. We borrow the method from ~\cite{dr3}.
\begin{lem}
\label{lem2} If $\phi$ is a smooth, then
\begin{equation}
\label{phiboundH} \int_{r\leq
R}\left(\frac{\phi}{1+r}\right)^2dx+\int_{S_{\tau}}\left(\frac{\phi}{1+r}\right)^2r^2dvd\om
\leq 6\tilde{E}[\phi](\tau).
\end{equation}
In particular
\begin{equation}
 \label{phiboundn}
\int_{r\leq R}\phi^2dx\leq 6(1+R)^2\tilde{E}[\phi](\tau).
\end{equation}
\end{lem}
\begin{proof}
Take a function $\eta$ as follows
$$\eta(r)=r-2\ln(1+r)+\frac{r}{1+r}.
$$
Then
$$\eta'(r)=\frac{r^2}{(1+r)^2}, \quad \eta(0)=0 , \quad |\eta(r)|\leq
r.
$$
Denote $d\si$ as $dx$ when $r\leq R$ and $r^2dvd\om$ when $r\geq R$.
Integration by parts and using lemma ~\ref{lem1} imply that
\begin{align}
\notag
 \int_{\Si_\tau}\left(\frac{\phi}{1+r}\right)^2d\si&=\int_{\om}\int_{0}^{R}\phi^2 d \eta d\om+\int_{\om}\int_{v_\tau}^{\infty}\phi^2 d \eta d\om\\
\notag
                                             &=\left.\int_{\om}\phi^2 \eta d\om\right|_{0}^{\infty}-2\int_{r\leq R} \eta \phi\cdot \phi_r dr d\om-2\int_{S_\tau} \eta \phi\cdot \phi_v dv d\om\\
\notag
                                             &\leq \tilde{E}[\phi](\tau) +\f12 \int_{\Si_\tau} \eta^2 r^{-4}\phi^2d\si+2\int_{r\leq R}\phi_r^2dx +2\int_{S_\tau} \phi_v^2 d\si\\
\label{phiboundH1}
                                              &\leq 3\tilde{E}[\phi](\tau)+ \f12\int_{\Si_\tau}\eta^2
                                              r^{-4}\phi^2d\si,
\end{align}
where $r=v-u$ and on $S_\tau$, $u$ is constant.

Notice that $\eta$ is nonnegative and $\ln(1+r)\geq\frac{r}{1+r}$.
We conclude that
$$\frac{\eta}{r}=\frac{r}{1+r}-2\left(\frac{\ln(1+r)}{r}-\frac{1}{1+r}\right)\leq\frac{r}{1+r}.$$
Hence
$$\int_{\Si_\tau}\eta^2 r^{-4}\phi^2d\si\leq \int_{\Si_\tau}\left(\frac{\phi}{1+r}\right)^2d\si
$$
The inequality ~\eqref{phiboundH} then follows from
\eqref{phiboundH1} by absorbing the second term. Inequality
~\eqref{phiboundn} follows from ~\eqref{phiboundH} if we restrict
the integral in the region $r\leq R$.
\end{proof}

In the region $r\geq R$, we analyze weighted solution $\psi=r\phi$ instead of $\phi$ itself. In the energy level, they are equivalent in the
sense of the following corollary.

\begin{cor}
 \label{cor2}
In the outer region $r\geq R$, we have
\begin{equation}
\label{phipsieq} \left|\int_{S_\tau}|\pa_v(r\phi)|^2+|\nabb
(r\phi)|^2dvd\om-
\int_{S_\tau}\left((\pa_v\phi)^2+|\nabb\phi|^2\right)r^2dvd\om
\right|\leq 2\tilde{E}[\phi](\tau).
\end{equation}
\end{cor}
\begin{proof} In fact, since
$$
\int_{S_\tau}(r\phi)_v^2dvd\om= \int_{S_\tau}\left [r^2 \phi_v^2
+(r\phi^2)_v\right]dvd\om = \int_{S_\tau} \phi_v^2 d\si +\left.
\int_{\om}r \phi^2\right|_{R}^{\infty},
$$
the corollary follows from Lemma ~\ref{lem1}.
\end{proof}

\bigskip

Now we are able to estimate the boundary term in ~\eqref{menergyeq}.

\begin{prop}
\label{prop1}
Suppose $f$ and $\chi$ satisfy $$|f|\leq C_1, \quad|\chi|\leq
\frac{C_1}{1+r}, \quad|\chi'|\leq \frac{C_1}{(1+r)^{2}}$$
for some constant $C_1$. Then
\begin{equation*}
\left|\int_{\Sigma_\tau}\tilde{J}_{\mu}^X[\phi]n^{\mu}
d\si\right|\leq 6C_1(1+H+2H^2)^4\tilde{E}[\phi](\tau).
\end{equation*}
\end{prop}
\begin{proof} For $r\leq R$, notice that on the surface $t=\tau$, we have
\begin{align*}
 \tilde{J}_\mu^{X}[\phi]n^{\mu}d\si&=-\tilde{J}_{\mu}^{X}[\phi]g^{t\mu}\sqrt{-G}dx\\
   &=-\left(f\pa^{t}\phi\pa_r\phi-\f12\pa^{t}\chi\cdot\phi^2+\f12\chi\cdot\pa^{t}\phi^2\right)\sqrt{-G}dx.
\end{align*}
Since $|g^{\a\b}|\leq 1+H$ and $|g^{\a\b}|\leq H, \a\neq\b$, we have
\begin{equation}
\label{MaxG}
 \sqrt{-G}\leq \sqrt{(1+H+2H^2)^4}= (1+H+2H^2)^2
\end{equation}
and
\begin{align*}
 |f\pa^{t}\phi\pa_r\phi|=|fg^{t\mu}\pa_{\mu}\phi\pa_r\phi|\leq
 \frac{1+5H}{2}C_1|\pa\phi|^2.
\end{align*}
Under the assumptions on $\chi$, we can estimate
$$|\f12\pa^{t}\chi\cdot\phi^2|=|\f12 g^{t\mu}\pa_{\mu}\chi\cdot\phi^2|\leq \sum\limits_{i=1}^{3}\f12 H|\chi'|\frac{|x_i|}{r}\phi^2\leq HC_1\frac{\phi^2}{(1+r)^2}
$$
and
$$|\f12\chi\cdot\pa^{t}\phi^2|=|\chi\phi g^{t\mu}\pa_{\mu}\phi|\leq C_1\frac{1+4H}{2}\frac{\phi^2}{(1+r)^2}+
C_1\frac{1+H}{2}|\pa\phi|^2.
$$
For $r\geq R$, the metric is flat according to our assumption. The
unit normal vector field to $S_\tau$ is
$$n=\frac{1}{\sqrt{2}}\pa_v=\frac{1}{\sqrt{2}}(\pa_t+\pa_r).$$
Thus we can calculate
$$\tilde{J}_{\mu}^X[\phi]n^{\mu}d\si=\left(\f12 f(|\pa_v\phi|^2 - |\nabb\phi|^2)-\f12\chi'\cdot\phi^2 +
\chi\cdot\pa_v\phi\cdot\phi \right)r^2dvd\om.$$ On the other hand,
the condition on $\chi$ shows that
\begin{align*}
&|\f12 \chi'\phi^2|\leq\frac{C_1}{2}\left(\frac{\phi}{1+r}\right)^2,\quad |\chi\pa_v\phi\cdot\phi|\leq\frac{C_1}{4}\left(\frac{\phi}{1+r}\right)^2+C_1(\pa_v\phi)^2.
\end{align*}
Therefore, according to ~\eqref{phiboundH}, we can show
\begin{align*}
 \left|\int_{\Si_\tau}\tilde{J}_\mu^{X}[\phi]n^{\mu}d\si\right|
&\leq C_1(1+3H)\int_{r\leq R}|\pa\phi|^2\sqrt{-G}dx+\frac{3C_1}{2}\int_{S_\tau}\left(|\pa_v\phi|^2+|\nabb\phi|^2\right)r^2dvd\om\\
&\quad+C_1\frac{1+6H}{2}\int_{r\leq R}\left(\frac{\phi}{1+r}\right)^2\sqrt{-G}dx+\frac{3C_1}{4}\int_{S_\tau}\left(\frac{\phi}{1+r}\right)^2r^2dvd\om\\
&\leq \frac{3C_1}{2}(1+H+2H^2)^4\tilde{E}[\phi](\tau)+
\frac{3C_1}{4}(1+H+2H^2)^4\cdot6\tilde{E}[\phi](\tau)\\
&\leq 6C_1(1+H+2H^2)^4\tilde{E}[\phi](\tau).
\end{align*}
Hence we conclude the proposition.
\end{proof}

\begin{remark}
 If $\tilde{E}[\phi](\tau)$ is finite, all the above statements hold if we
 replace $\tilde{E}[\phi](\tau)$ with $E[\phi](\tau)$. The reason is as follows: under the
 assumption that $\tilde{E}[\phi](\tau)$ is finite, Lemma ~\ref{lem1} holds for $E[\phi](\tau)$. Hence 
Lemma ~\ref{lem2}, Corollary ~\ref{cor2} and Proposition ~\ref{prop1} also hold if we replace
 $\tilde{E}[\phi](\tau)$ with $E[\phi](\tau)$.
\end{remark}

We hope that, by using the multiplier method, we can derive the $H^1$ estimates of the solution on $\Si_\tau$. That is we
want to show that the energy current $J^T_\mu[\phi]n^\mu$ has a positive sign. 
The proposition below guarantees this positive sign provided that $(g^{ij})$ is uniformly elliptic. In addition, the elliptic method we will use later on also requires the uniform ellipticity of $(g^{ij})$. This is obvious when $(\mathbb{R}^{3+1}, g)$ is a small perturbation of Minkowski. We claim that it still holds under the condition $\mathcal{A}1$.

\begin{prop}
 \label{prop2}
If $g$ satisfies $\mathcal{A}1$, then there is a constant $\lambda_1$ such that
\begin{equation*}
\lambda_1\leq-g^{00}\leq \lambda_1^{-1}, \quad \lambda_1
I_{3\times3}\leq(g^{ij})\leq \lambda_1^{-1}I_{3\times3},
\end{equation*}
where $(g^{\mu\nu})=(g_{\mu\nu})^{-1}$ and $\mu, \nu$ run from 0 to 3 and $i,j$ run from 1 to 3.
\end{prop}

From the geometric point of view, the hypersurface $t=const$ is
spacelike. Therefore its normal is timelike. Since we require
$\pa_t$ to be timelike, the positivity of the current
$J^T_\mu[\phi]n^\mu$ follows from the fact that $\mathbb{T}(X, Y)$
is positive for any two timelike vector fields $X, Y$. However, we
are concerned about the uniform lower bound. We prefer the following
algebraic proof of Proposition ~\ref{prop2}.

\begin{proof}
We write the matrix $g$ as
 \[(g_{\mu\nu})=\left(
\begin{array}{cc}
 -a&b\\
b^T &D
\end{array}
\right),\] where $a=-g_{00}$, $b=(g_{01}, g_{02}, g_{03})$ and
$D=(g_{ij})$. Then $G=det (g_{\mu\nu})= -(a+bD^{-1}b^T)det D $ and
 \[(g^{\mu\nu})=\left(
\begin{array}{cc}
 -a&b\\
b^T &D
\end{array}
\right)^{-1}=\left(
\begin{array}{cc}
 -\frac{1}{a+bD^{-1}b^T}&\frac{bD^{-1}}{a+bD^{-1}b^T}\\
\frac{D^{-1 }b^T}{a+bD^{-1}b^T} &D^{-1}-\frac{D^{-1}b^TbD^{-1}}{a+bD^{-1}b^T}
\end{array}
\right).\] Since $g$ satisfies condition $\mathcal{A}1$, we have
$-a=-g_{00}\geq \lambda$ and $D$ is positive definite. Thus
$$\lambda \leq a +bD^{-1}b^T\leq 1+H +\lambda^{-1}\|b\|^2\leq 1+H+3\lambda^{-1}
H^2.
$$
Denote $X=bD^{-1}$. For any vector $Y$, we have two cases:
\begin{itemize}
 \item[]If $Y$ is perpendicular to $X$, then
 $$\lambda \|Y\|^2\leq Y\left(D^{-1}-\frac{D^{-1}b^TbD^{-1}}{a+bD^{-1}b^T}\right)Y^T\leq \lambda^{-1} \|Y\|^2.$$
\item[] If $Y=X$, then
$$Y\left(D^{-1}-\frac{D^{-1}b^TbD^{-1}}{a+bD^{-1}b^T}\right)Y^T=\frac{a XD^{-1}X^T}{a+bD^{-1}b^T}+
\frac{XD^{-1}X^TXDX^T-|X|^4}{a+bD^{-1}b^T}.
$$
Notice that
$$XD^{-1}X^TXDX^T\geq|X|^4,$$
which follows from Cauchy's inequality if we assume, without loss of
generality, that $D$ is diagonal. Therefore we have
$$ \frac{\lambda ^2 \|X\|^2}{1+H+3\lambda^{-1}H^2}\leq Y\left(D^{-1}-\frac{D^{-1}b^TbD^{-1}}{a+bD^{-1}b^T}\right)
Y^T\leq \lambda^{-1}\|X\|^2.$$
\end{itemize}
Now $g^{00}=-\frac{1}{a+bD^{-1}b^T}$ and $(g^{ij})=D^{-1}-\frac{D^{-1}b^TbD^{-1}}{a+bD^{-1}b^T}$, the Proposition 
then follows if we choose $\lambda_1=\frac{\lambda^2}{1+H+3H^2\lambda^{-1}}$.

\end{proof}

Having proven this proposition, we conclude that the energy flux
through $\Si_\tau$ in $(\mathbb{R}^{3+1}, g)$ is equivalent to that
in Minkowski space.

\begin{cor}
 \label{corollary3}
If $g$ satisfies condition $\mathcal{A}1$, then there is a constant
$\lambda_2$ such that
\begin{equation*}
\lambda_2 E[\phi](\tau)\leq
\int_{\Si_\tau}J^T_\mu[\phi]n^{\mu}d\si\leq
\lambda_2^{-1}E[\phi](\tau).
\end{equation*}
\end{cor}
\begin{proof} In $(t, x)$ coordinates, on $\{|x|\leq R\}\cap \Si_\tau$, we have
\begin{align*}
J^T_\mu[\phi]n^\mu d\si&=-J^T_\mu[\phi]g^{t\mu}\sqrt{-G}dx\\
                       &=\f12\left(\pa^i\phi\pa_i\phi-\pa^t\phi\pa_t\phi\right)\sqrt{-G}dx\\
                       &=\f12\left(g^{ij}\pa_i\phi\pa_j\phi-g^{00}|\pa_t\phi|^2\right)\sqrt{-G}dx.
\end{align*}
Proposition \ref{prop2} then implies that
\[
 \lambda_1\sum\limits_{\b=0}^3|\pa_\b\phi|^2\leq g^{ij}\pa_i\phi\pa_j\phi-g^{00}\pa_t^2\phi\leq
 \lambda^{-1}\sum\limits_{\b=0}^3|\pa_\b\phi|^2.
\]
On one hand, using the notations in Proposition \ref{prop2}, we have
the lower bound of $-G$
$$-G=(a+bD^{-1}b^T)det D\geq \lambda^4.$$
On the other hand, we have already shown in ~\eqref{MaxG} that
$$|G|\leq (1+H+2H^2)^4.$$
Let $\lambda_2=\frac{\lambda_1\lambda^2}{2(1+3H)^2}$. We obtain
\[
 \lambda_2\int_{r\leq R}\sum\limits_{\b=0}^3|\pa_\b\phi|^2dx\leq\int_{r\leq R}J^T_\mu[\phi]n^\mu d\si
 \leq \lambda_2^{-1}\int_{r\leq
 R}\sum\limits_{\b=0}^3|\pa_\b\phi|^2dx.
\]
Notice that the metric is flat on $S_\tau$. Thus we have
\[
 J^T_\mu[\phi]n^\mu
 d\si=\f12\left(|\pa_v\phi|^2+|\nabb\phi|^2\right)r^2dvd\om.
\]
Since $\lambda_2<\frac{1}{2}$, the Corollary then follows.
\end{proof}

\bigskip

The proof of Theorem 1 is quite similar to that of Theorem 2. We
thus prove them together. For Theorem ~\ref{thm2}, we prove it by
taking
\[
 \ep_0=\sup\limits_{\b<1} \frac{\b}{700(1+\f12 R)^{\b+1}}.
\]
Thus for $H<\ep_0$, there exists a positive constant $\a<1$ such
that
\begin{equation}
\label{smallnesscond}
H\leq \frac{\a}{700(1+\f12 R)^{\a+1}}.
\end{equation}
From now on, we instead use condition \eqref{smallnesscond} on $H$.

\bigskip

Having the above basic preparations, we are now able to establish the key estimates in this paper.
\begin{prop}
\label{ILEthm}
If $g$ satisfies the conditions $\mathcal {A}1$ and $\mathcal{A}2$ or the smallness assumption ~\eqref{smallnesscond}, then there is a constant $C$ depending on $H, R, \a$, $\lambda$, $C_0$ such that
\begin{itemize}
\item[(1)] Integrated energy bound
\begin{equation}
 \label{ILE0}
\int_{\tau_1}^{\tau_2}\int_{\Si_\tau}\frac{\phi_t^2+|\pa_r\phi|^2+|\nabb\phi|^2}{(1+r)^{\a+1}}+\frac{\phi^2}{r(1+r)^{\a+2}}dxd\tau
\leq C\tilde{E}[\phi](\tau_1)+CD[F]_{\tau_1}^{\tau_2}.
\end{equation}
In particular, for the smallness assumption \eqref{smallnesscond}, $C$ can be $\frac{36^2}{\a^2}$.
\item[(2)] Energy bound
\begin{equation}
 \label{eb}
\tilde{E}[\phi](\tau_2)+I[\phi]_{\tau_1}^{\tau_2}\leq
C\tilde{E}[\phi](\tau_1)+CD[F]_{\tau_1}^{\tau_2}.
\end{equation}
\item[(3)] Improved integrated energy bound in angular directions
\begin{equation}
 \label{bILE} \int_{\tau_1}^{\tau_2}\int_{S_\tau}\frac{|\nabb\phi|^2}{r}dxd\tau\leq C\tilde{E}[\phi](\tau_1)
 +CD[F]_{\tau_1}^{\tau_2}.
\end{equation}
\end{itemize}
Here recall that
$D[F]_{\tau_1}^{\tau_2}=\int_{\tau_1}^{\tau_2}\int_{\Si_\tau}|F|^2(1+r)^{\a+1}dxd\tau$.
\end{prop}
We mention here that variants and generalizations of estimate
\eqref{ILE0} can also be found in \cite{sogge-metcalfe2},
\cite{sogge-metcalfe}. To be consistent, we give the proof of the
above proposition by following the method in \cite{dr3}.

\begin{proof} We choose the function $f, \chi$ suitably to make the coefficients on the right hand
 side of \eqref{menergyeq} positive. Take $\a$ exactly as the one in condition $\mathcal{A}2$ or
 ~\eqref{smallnesscond} and $\b=\frac{2}{\a}$.
  Set $$f=\b-\frac{\beta}{(1+r)^{\alpha}},\quad \chi=r^{-1}f.$$
Direct calculations show that
\begin{align*}
 &r^{-1}f + \f12 f' - \chi=\frac{1}{(1+r)^{\alpha +1}},\\
&\chi'=-\b\frac{(1+r)^{\a+1}-(\a+1)r-1}{r^2(1+r)^{\a+1}},\\
&\chi\prime\prime=2\b r^{-3}-\frac{2\b}{r^3(1+r)^{\a}}-\frac{4}{r^2(1+r)^{\a+1}}-\frac{2(\a+1)}{r(1+r)^{\a+2}},\\
&\pa_{ij}\chi=-\frac{2(\a+1)}{r(1+r)^{\a+2}}\cdot\frac{x_i
x_j}{r^2}+\chi'\left(\frac{\delta_{ij}}{r}- \frac{3x_i
x_j}{r^3}\right).
\end{align*}
Notice that
$$\frac{\a}{1+r}\leq\frac{(1+r)^{\a}-1}{r}\leq(1+r)^{\a-1}.
$$
We obtain
\begin{align*}
&\chi=\frac{\b}{1+r}\frac{1+r}{r}\frac{(1+r)^{\a}-1}{(1+r)^{\a}}\leq \frac{\b}{1+r},\\
&\chi-r^{-1}f + \f12 f'=\frac{\alpha\beta-1}{(1+r)^{\alpha+1}}=\frac{1}{(1+r)^{\alpha+1}},\\
&\chi- \f12 f'=r^{-1}\b - \frac{\beta(1+r)+r}{r(1+r)^{\alpha +1}}=
\frac{\b\left((1+r)^\a-1\right)}{r(1+r)^{\a}}-\frac{1}{(1+r)^{1+\a}}\geq\frac{1}{(1+r)^{\a+1}}.
\end{align*}
To apply Proposition ~\ref{prop1}, we need to estimate $\chi'$.
Since
$$0\leq\frac{(1+r)^{\a+1}-(\a+1)r-1}{r^2}=\frac{(1+r)^{\a}-1}{r}+\frac{(1+r)^{\a}-\a
r-1}{r^2}\leq(1+r)^{\a-1},
$$
we conclude that
$$|\chi'|=-\chi'\leq\frac{\b}{(1+r)^{\a+1}}\cdot(1+r)^{\a-1}=
\frac{\b}{(1+r)^{2}}.
$$
Notice that $|f|\leq \b$. We have shown that $f, \chi$ satisfy
conditions in Proposition ~\ref{prop1}.

\bigskip

We have two cases according to the conditions $g$ satisfies.
\begin{itemize}
\item[Case1]: If $h$ satisfies the smallness condition ~\eqref{smallnesscond}, then
\begin{align*}
-\Box_g\chi&=-\frac{1}{\sqrt{-G}}\pa_\a\left(g^{\a\b}\sqrt{-G}\pa_\b\chi\right)\\
           &=-g^{ij}\pa_{ij}\chi-\left(\pa_\a g^{\a i}+\f12 g^{\a i}\pa_{\a}g_{\b\gamma}\cdot g^{\b\gamma}\right)\pa_i\chi\\
           &\geq -\Delta \chi - H |\pa_{ij}\chi|-\left(4H+ 32
           H(1+H)^2\right)\frac{|x_i|}{r}|\chi'|.
\end{align*}
As having computed above, we find that
$$-\Delta \chi=-\chi\prime\prime -\frac{2}{r}\chi'=\frac{2(\a+1)}{r(1+r)^{\a+2}},$$

$$|\pa_{ij}\chi|\leq \frac{2(\a+1)}{r(1+r)^{\a+2}}+|\chi'|\frac{2}{r}\leq
\frac{2(\a+\b+1)}{r(1+r)^{\a+2}}.
$$

Recall that $H\leq \frac{\a}{700 (1+\f12 R)^{\a+1}}$, $\a<1$ and $h$
is only supported in $r\leq \f12 R$. We conclude that
$$-\Box_g\chi\geq \frac{1}{r(1+r)^{\a+2}}.
$$

It remains to show that the error term ~\eqref{error} can be
absorbed provided that $H$ is small. In fact since
$$|f'|=\frac{2}{(1+r)^{\a+1}}\leq 2,\quad|r^{-1}f|=\frac{\b}{(1+r)^\a}\frac{(1+r)^\a-1}{r}
\leq\frac{\a\b}{(1+r)^\a}\leq2,$$ we have
$$|\pa_j(f\frac{x_i}{r})=|f'\frac{x_i x_j}{r^2}+f(\frac{\delta_{ij}}{r}-\frac{x_i x_j}{r^3})|\leq
2.
$$
Using Cauchy-Schwartz inequality, we obtain
\[
 \sum |\pa_\mu\phi\pa_i\phi|\leq\sum\limits_{i}(|\pa_t\phi|^2+\frac{1}{4}|\pa_i\phi|^2)+\frac{1}{2}\sum\limits_{i,j}|\pa_i\phi|^2+|\pa_j\phi|^2\leq \frac{13}{4}|\pa\phi|^2
\]
and
\[
 |\pa^\gamma\phi\pa_\gamma\phi+|\pa_t\phi|^2-|\pa_i\phi|^2|\leq|h^{\gamma\mu}\pa_\gamma\phi\pa_\mu\phi|\leq
 4|\pa\phi|^2.
\]
Therefore we can estimate
\begin{align*}
 |error|&\leq 6H\cdot\frac{13}{4}|\pa\phi|^2+4 H|\pa\phi|^2+4\b H  |\pa\phi|^2+2\b H(1+4H)^2|\pa\phi|^2\\
&\leq 18\b H |\pa\phi|^2\leq \frac{1}{18(1+\f12 R)^{\a+1}}((\pa_t\phi)^2+(\pa_r\phi)^2+|\nabb\phi|^2)\\
&\leq
\frac{1}{18(1+r)^{\a+1}}((\pa_t\phi)^2+(\pa_r\phi)^2+|\nabb\phi|^2),
\end{align*}
where we recall that that $\b>2$, $H\leq\frac{\a}{700(1+\f12
R)^{\a+1}}$ and the $error$ is supported in $r\leq \f12 R$.

On the null infinity $\mathcal I_{\tau_1}^{\tau_2}$, the unit normal
is
$$n=\frac{1}{\sqrt{2}}\pa_v=\frac{1}{\sqrt{2}}(\pa_t-\pa_r).$$ Thus we can write
$$\tilde{J}_{\mu}^X[\phi]n^{\mu}d\si=\f12\left(f(-(\pa_u\phi)^2 + |\nabb\phi|^2)-\chi'\cdot\phi^2 +
2\chi\cdot\pa_u\phi\cdot\phi\right)r^2dud\om.$$ Notice that at null
infinity, $\chi'=0$ and $f=\b$. The above inequality implies that
$$\left|\int_{\mathcal I_{\tau_1}^{\tau_2}}\tilde{J}_{\mu}^X[\phi]n^{\mu}d\sigma\right|\leq
\f12\b I[\phi]_{\tau_1}^{\tau_2}.$$

For $H\leq\frac{\a}{700(1+\f12 R)^{\a+1}}$, we have
\begin{equation}
 \label{MinG}
-G\geq (1-H)^4-6(1+H)^2H^2-8(1+H)H^3-9H^4\geq(1-H-2H^2)^4.
\end{equation}
Using Proposition ~\ref{prop1} and the inequality \eqref{menergyeq},
we obtain
\begin{align}
\label{menergyeq1}
&\int_{\tau_1}^{\tau_2}\int_{\Si_\tau}\frac{\phi_t^2+\phi_r^2+|\nabb\phi|^2}{(1+r)^{\a+1}}+\frac{\phi^2}{r(1+r)^{\a+2}}dxd\tau\\
\notag
 &\leq 7\b\left(\tilde{E}[\phi](\tau_1)+\tilde{E}[\phi](\tau_2) +\frac{1}{13}I[\phi]_{\tau_1}^{\tau_2}\right)+ 2\b\int_{\tau_1}^{\tau_2}\int_{\Si_\tau}
|F\phi_r|+|F\frac{\phi}{1+r}|dxd\tau.
\end{align}

\item[Case2]: If $g$ satisfies $\mathcal{A}2$, although we can not guarantee the positivity of $-\Box_g \chi$
as in case 1, we can control this term and the error term by using
the condition ~\eqref{morawetz1}. The error term can be estimated in
the same way as having shown in case 1. We thus obtain a similar
inequality to ~\eqref{menergyeq1} but with a constant multiple of $D[F]_{\tau_1}^{\tau_2}$
on the right hand side . Although this
constant depends on $\a$, $R$, $\lambda$, $h$, $C_0$, the
explicit constants are not essential for this case.

\end{itemize}
Now, we need to relate the energy flux on ${\Si}_{\tau_1}$ and that on
${\Si}_{\tau_2}$. For this purpose, take $X$ to be the vector field $T$.
We have
$$|K^T[\phi]|=|\f12\pa_t g^{\mu\nu}{\mathbb T}_{\mu\nu}|\leq\frac{H}{2}(4+2(1+4H)^2)|\pa\phi|^2\leq
 3H(1+H+2H^2)^3|\pa\phi|^2.$$
By ~\eqref{energyeq}, we obtain
\begin{align}
\notag
&\int_{{\Si}_{\tau_2}}J^T_{\mu}[\phi]n^{\mu}d\si+\int_{\mathcal I_{\tau_1}^{\tau_2}}J^T_{\mu}[\phi]n^{\mu}d\si
=\int_{{\Si}_{\tau_1}}J^T_{\mu}[\phi]n^{\mu}d\si-\int_{\tau_1}^{\tau_2}\int_{\Si_\tau}F\phi_t + K^T[\phi]d\vol\\
\label{energyineq}
&\leq\int_{{\Si}_{\tau_1}}J^T_{\mu}[\phi]n^{\mu}d\si+\int_{\tau_1}^{\tau_2}\int_{\Si_\tau}|F\phi_t|d\vol
+ 3H(1+H+2H^2)^5\int_{\tau_1}^{\tau_2}\int_{r\leq \f12
R}|\pa\phi|^2dxd\tau.
\end{align}
Notice that
\begin{equation}
 \label{energyexpr}
J^T_{\mu}[\phi]n^{\mu}d\si=\f12\left(\pa^{i}\phi\pa_i\phi-\pa^{t}\phi\pa_t\phi\right)\sqrt{-G}dx.
\end{equation}
Again we distinguish two cases.
\begin{itemize}
 \item[Case1]:When $H$ satisfies ~\eqref{smallnesscond}, we can show
 \begin{align*}
|\pa^{i}\phi\pa_i\phi-\pa^{t}\phi\pa_t\phi-(\pa_\mu\phi)^2| &= |g^{ij}\pa_i\phi\pa_j\phi-g^{00}
(\pa_t\phi)^2-|\pa\phi|^2|\\
&\leq H|\pa\phi|^2+2H(\pa_i\phi)^2\leq 3H|\pa\phi|^2.
\end{align*}
Using the facts that the metric is flat outside the cylinder $\{(t,
x)||x|\leq \f12 R\}$ and $H\leq \frac{1}{700}$, inequalities
~\eqref{MaxG}, ~\eqref{MinG} and ~\eqref{energyexpr} imply that
$$\frac{13}{14}E[\phi](\tau)\leq2\int_{\Si_{\tau}}J^T_{\mu}[\phi]n^{\mu}d\si\leq
\frac{13}{12}E[\phi](\tau).
$$
Recall that
$\tilde{E}[\phi](\tau)=E[\phi](\tau)+I[\phi]_{0}^{\tau}$. By
~\eqref{energyineq}, we have
\begin{align*}
&\quad\quad\tilde{E}[\phi](\tau_2)+\frac{1}{13}I[\phi]_{\tau_1}^{\tau_2}\leq
\frac{7}{6}\tilde{E}[\phi]
(\tau_1)+\frac{13}{6}\int_{\tau_1}^{\tau_2}\int_{\Si_\tau}|F\phi_t|dxd\tau+\frac{20}{3}H\int_{\tau_1}^{\tau_2}
\int_{r\leq \f12 R}(\pa\phi)^2 dxd\tau.
\end{align*}
Restricting the integral of the left hand side of
~\eqref{menergyeq1} to the region $r\leq \f12 R$, we can estimate
the above inequality as follows
\begin{align*}
&\leq \frac{7}{6}\tilde{E}[\phi](\tau_1)+\frac{13}{6}\int_{\tau_1}^{\tau_2}\int_{\Si_\tau}|F\phi_t|dxd\tau+\frac{20}{3}\frac{\a}{700}7\b\left(\tilde{E}[\phi](\tau_1)+\tilde{E}[\phi](\tau_2)\right)\\
&\quad +\frac{20}{3}\frac{\a}{700}2\b\int_{\tau_1}^{\tau_2}\int_{\Si_\tau}
|F\phi_r|+|F\frac{\phi}{1+r}|dxd\tau+\frac{20}{3}\frac{\a}{700}\frac{7}{13}\b I[\phi]_{\tau_1}^{\tau_2}\\
&\leq \frac{4}{3}\tilde{E}[\phi](\tau_1)+\frac{2}{15}\left(\tilde{E}[\phi](\tau_2)+\frac{1}{13}I[\phi]_{\tau_1}^{\tau_2}\right)+\frac{13}{6}\int_{\tau_1}^{\tau_2}\int_{\Si_\tau}|F\phi_t|dxd\tau\\
&\quad+\frac{4}{105}\int_{\tau_1}^{\tau_2}\int_{\Si_\tau}|F\phi_r|+\frac{|F\phi|}{1+r}dxd\tau,
\end{align*}
which implies that
\begin{equation}
\label{eb0}
\begin{split}
\tilde{E}[\phi](\tau_2)+\frac{1}{13}I[\phi]_{\tau_1}^{\tau_2}&\leq 2\tilde{E}[\phi](\tau_1)+\frac{5}{2}\int_{\tau_1}^{\tau_2}\int_{\Si_\tau}|F\phi_t|dxd\tau\\
&\quad+\frac{1}{20}\int_{\tau_1}^{\tau_2}\int_{\Si_\tau}|F\phi_r|+\frac{|F\phi|}{1+r}dxd\tau.
\end{split}
\end{equation}
Together with ~\eqref{menergyeq1}, we can show that
\begin{align*}
&\int_{\tau_1}^{\tau_2}\int_{\Si_\tau}\frac{\phi_t^2+\phi_r^2+|\nabb\phi|^2}{(1+r)^{\a+1}}+\frac{\phi^2}{r(1+r)^{\a+2}}dxd\tau\\
 &\leq 21\b \tilde{E}[\phi](\tau_1) + \frac{5}{2}\b\int_{\tau_1}^{\tau_2}\int_{\Si_\tau}7|F\phi_t|+
|F\phi_r|+|F\frac{\phi}{1+r}|dxd\tau.
\end{align*}
Using Cauchy-Schwartz inequality to $|F\phi_r|$, $|F\phi_t|$,
$F\frac{\phi}{1+r}$, for example
$$\frac{35\b}{2}|F\phi_t|\leq
\f12\left(\frac{35}{2}\b\right)^2|F|^2(1+r)^{\a+1}+\frac{1}{2}\frac{\phi_t^2}{(1+r)^{\a+1}},
$$
we can show that
\begin{equation}
\label{ILEsmall}
 \begin{split}
  &\int_{\tau_1}^{\tau_2}\int_{\Si_\tau}\frac{\phi_t^2+\phi_r^2+|\nabb\phi|^2}{(1+r)^{\a+1}}+\frac{\phi^2}{r(1+r)^{\a+2}}dxd\tau\\
 &\leq 42\b \tilde{E}[\phi](\tau_1) +\left( \left(\frac{35}{2}\b\right)^2+2(3\b)^2\right)D[\phi]_{\tau_1}^{\tau_2}\\
&\leq 42\b \tilde{E}[\phi](\tau_1)+(18\b)^2
D[\phi]_{\tau_1}^{\tau_2}.
 \end{split}
\end{equation}
Hence ~\eqref{ILE0} holds with constant $C=18^2\b^2$ for this case.
This explicit constant will be used later on. The energy inequality
~\eqref{eb} follows from ~\eqref{eb0}, ~\eqref{ILEsmall}.

\item[Case2]: When $g$ satisfies conditions $\mathcal{A}1$ and $\mathcal{A}2$, condition
\eqref{morawetz1} together with Corollary \ref{corollary3} and estimates
~\eqref{energyineq} lead to a similar inequality of \eqref{eb0} but
with a constant multiple of $D[F]_{\tau_1}^{\tau_2}$ on the right
hand side. This constant also depends on $C_0$(in
$\mathcal{A}2$), $H$, $\lambda_2$(in Corollary ~\ref{corollary3}).
However, once we have shown ~\eqref{menergyeq1} and ~\eqref{eb0},
the estimates ~\eqref{ILE0}, ~\eqref{eb} follow similarly as
in case 1.
\end{itemize}

We hence have proven ~\eqref{ILE0}, ~\eqref{eb} in any case.
Finally, the estimate ~\eqref{bILE} follows from the fact that
$$\chi- \f12 f'=r^{-1}\b - \frac{\beta(1+r)+r}{r(1+r)^{\alpha +1}}\geq \frac{R}{(1+R)^{\a+1}} \frac{1}{r}, \quad r\geq
R.
$$
\end{proof}

It will be much more convenient if we use $E[\phi](\tau)$ instead of
$\tilde{E}[\phi](\tau)$ in the following argument. We use
$\tilde{E}[\phi](\tau)$ in order to make the previous argument
rigorous. However, as we have pointed out in Remark 1, all the
statements above hold if we replace $\tilde{E}[\phi](\tau)$ with
$E[\phi](\tau)$ provided that $\tilde{E}[\phi](\tau)$ is finite. The
idea is that under the bootstrap assumptions, we first show that
$\tilde{E}[\phi](\tau)$ is finite. And then we conclude that
Proposition \ref{ILEthm} holds if we replace $\tilde{E}[\phi](\tau)$
with $E[\phi](\tau)$.

\begin{cor}
\label{cor4} Proposition ~\ref{ILEthm} holds if we replace
$\tilde{E}[\phi](\tau)$ with $E[\phi](\tau)$.
\end{cor}
\begin{proof} If $D[F]^{\tau}_{0}$ is finite, then $\tilde{E}[\phi](\tau)$ is finite according to ~\eqref{eb}.
 Thus by Remark 1, we infer that Proposition ~\ref{ILEthm} is also true if we replace $\tilde{E}[\phi](\tau)$
 with $E[\phi](\tau)$.

If $D[F]_{0}^{\tau}$ is infinite for some $\tau$, then Propositiuon
~\ref{ILEthm} holds automatically if we replace
$\tilde{E}[\phi](\tau)$ with $E[\phi](\tau)$.
\end{proof}

Restricting  the left hand side of ~\eqref{ILE0} to the region $r\leq R$, we obtain an integrated local energy inequality.
\begin{cor}
If $g$ satisfies ~\eqref{smallnesscond}, then
\begin{equation}
 \label{ILE}
\int_{\tau_1}^{\tau_2}\int_{r\leq
R}\phi_t^2+|\pa_r\phi|^2+|\nabb\phi|^2+\frac{\phi^2}{r}dxdt \leq
CE[\phi](\tau_1)+C D[F]_{\tau_1}^{\tau_2},
\end{equation}
where $C=18^2\b^2(1+R)^{\a+2}$, $\b=\frac{2}{\a}$. In particular
\begin{equation}
 \label{ILEH2}
\int_{\tau_1}^{\tau_2}\int_{r\leq
R}\phi_t^2+|\pa_r\phi|^2+|\nabb\phi|^2dxdt \leq
C_2E[\phi](\tau_1)+C_2 D[F]_{\tau_1}^{\tau_2}
\end{equation}
with $C_2=18^2\b^2(1+R)^{\a+1}$.
\end{cor}
If $g$ satisfies ~\eqref{smallnesscond}, we can show that $g$
satisfies the condition $\mathcal{A}1$. The corollary implies that
$g$ also satisfies the condition ~\eqref{morawetz1}.

\section{Weighted Energy Inequality}
In this section, we revisit the p-weighted energy inequality in a neighborhood of
 the null infinity originally developed by M. Dafermos and I. Rodnianski in ~\cite{newapp}. Starting
 from the integrated local energy inequality ~\eqref{ILE}, the p-weighted energy inequalities
 allow us to obtain the decay of the energy flux $E[\phi](\tau)$. Moreover, the p-weighted energy inequailites play 
an important role in estimating the quadratic nonlinearity with null condition.

\bigskip

To avoid too many constants, we use the notation $A\les B$ to indicate that there is a constant $C$, depending on $R$, $\a$, $\lambda$, $h$ and $C_0$, such that $A\leq C B$.
\begin{prop}
For any $0<p\leq 2$, we have the p-weighted energy inequality
 \begin{align}
\notag
&\int_{S_{\tau_2}}r^p\psi_v^2dv d\om+\int_{\tau_1}^{\tau_2}\int_{S_\tau}r^{p-1}(\psi_v^2+|\nabb\psi|^2)dvd\om d\tau\\
\label{pWEineq} &\les E[\phi](\tau_1)+
\int_{S_{\tau_1}}r^p\psi_v^2dv d\om +
\int_{\tau_1}^{\tau_2}\int_{S_\tau}r^{p+1}
|F|^2d\vol+D[F]_{\tau_1}^{\tau_2}, \quad \psi=r\phi,
\end{align}
where the implicit constant also depends on $p$.
\end{prop}

\begin{proof}  When $r\geq R$, the metric is flat. Rewrite the equation \eqref{THEWAVEEQ} in null coordinates
\begin{equation}
\label{waveqpsi} -\pa_u \pa_v \psi+\lap \psi=rF,\quad \psi:=r\phi.
\end{equation}
Multiplying the equation by $r^p \pa_v\psi$ and then integration by
parts in the region bounded by the two null hypersurfaces
$S_{\tau_1}, S_{\tau_2}$ and the hypersurface $\{(t, x)||x|=R\}$, we
obtain
\begin{align}
\notag
&\int_{S_{\tau_2}} r^p (\pa_v\psi)^2 dvd\om +\int_{\tau_1}^{\tau_2}\int_{S_\tau}2r^{p+1}F\cdot\pa_v\psi dvd\tau d\om\\
 \notag& +\int_{\tau_1}^{\tau_2}\int_{S_\tau} r^{p-1}  \left (p(\pa_v\psi)^2 +
(2-p) |\nabb\psi|^2\right)dvd\tau d\om +\int_{\mathcal I_{\tau_1}^{\tau_2}} r^p |\nabb\psi|^2 du d\om\\
=& \int_{S_{\tau_1}}r^p (\pa_v\psi)^2 dvd\om +\int_{\tau_1}^{\tau_2} r^p \left (|\nabb\psi|^2- (\pa_v\psi)^2\right)d\om d\tau |_{r=R}.
\label{pWE}
\end{align}
We claim that we can estimate the boundary term on $\{r=R\}$ as
follows
\begin{align*}
\left|\int_{\tau_1}^{\tau_2} r^p \left (|\nabb\psi|^2- (\pa_v\psi)^2\right)d\om d\tau |_{r=R}\right|&=R^p\left|\int_{\tau_1}^{\tau_2}\left (|\nabb\psi|^2- (\pa_v\psi)^2\right)d\om d\tau\right|\\
&\les D[F]_{\tau_1}^{\tau_2}+E[\phi](\tau_1).
\end{align*}
In fact, it suffices to consider the case for $p=0$. Thus set $p=0$ in the above p-weighted energy inequality \eqref{pWE}.
Notice that
\begin{align*}
\int_{\tau_1}^{\tau_2}\int_{S_\tau}2r|F|\cdot|\pa_v\psi| dvd\tau d\om & \leq\int_{\tau_1}^{\tau_2}\int_{S_\tau}|F|^2(1+r)^{\a+1} +2\frac{r^2\phi_v^2 + \phi^2}{r^2(1+r)^{\a+1}}d\vol\\
&\les E[\phi](\tau_1)+D[F]_{\tau_1}^{\tau_2}
\end{align*}
and
 $$\int_{\mathcal I_{\tau_1}^{\tau_2}} r^p |\nabb\psi|^2 du d\om=\int_{\mathcal I_{\tau_1}^{\tau_2}}
  |\nabb\phi|^2 r^2du d\om\leq I[\phi]_{\tau_1}^{\tau_2}.$$
Inequalities \eqref{phipsieq}, \eqref{eb} and \eqref{bILE} then
yield the desired estimates. Here we recall that, by Corollary
~\ref{cor4}, Proposition ~\ref{ILEthm} holds if we replace
$\tilde{E}[\phi](\tau)$ with $E[\phi](\tau)$.

\bigskip

For general $0<p\leq 2$, we can estimate the inhomogeneous term as
follows
$$r^{p+1}|F\pa_v\psi|\leq \frac{p}{2}r^{p-1}(\pa_v\psi)^2 +
\frac{2}{p}r^{p+3}|F^2|.
$$
Then the p-weighted energy inequality ~\eqref{pWEineq} follows by
absorbing the first term.
\end{proof}

\bigskip

Intuitively, the inequality \eqref{pWEineq} indicates that the
solution $\phi$, at least the good derivative of the solution
$\pa_v\phi$,
 decays fast in $r$ since $p\in[0, 2)$ and $\psi=r\phi$.
 However,
 for the nonlinear problem, merely decay in $r$ is not enough to obtain the long time existence
  of the solution. We thus have to transfer the decay in $r$ to the decay in time $t$. The first step
   to realize this is to find some quantities which decay in $t$. However, in this context, decay in time $t$ is characterized by 
the parameter $\tau$ of the foliation $\Si_\tau$. We show that, under appropriate assumptions on the nonlinearity $F$,
    the energy flux through the hypersurface $\Si_\tau$ decays in $\tau$ by combining
    the integrated local energy inequality and the p-weighted energy inequality.
\begin{prop}
\label{prop4}
If there is a constant $C_1$ such that $F$ satisfies the following conditions:
\begin{description}
\item[(a)] $\int_{\tau_1}^{\tau_2}\int_{S_\tau}|F|^2r^{3-\a}d\vol\leq C_1,$
\item[(b)] $\int_{\tau_1}^{\tau_2}\int_{S_\tau}|F|^2r^{2}d\vol\leq C_1(1+\tau_1)^{-1+\a},$
\item[(c)] $D[F]_{\tau_1}^{\tau_2}\les C_1(1+\tau_1)^{-2+\a}+E[\phi](\tau_1)$
\end{description}
for all $\tau_1\leq\tau_2$,  then we have the energy flux decay
$$E[\phi](\tau)\les \left(\ep^2 E_0+C_1\right)(1+\tau)^{-2+\a}.$$
\end{prop}
\begin{proof} We first take $p=2 - \a$ in the p-weighted
energy inequality ~\eqref{pWEineq}. Since the initial data are
supported in $\{|x|\leq R\}$, the finite speed of propagation for
wave equations shows that the solution $\phi$ vanishes on $S_0$. Let
$\tau_2=\tau$, $\tau_1=0$. Under our assumptions on $F$, we obtain
\begin{align}
\label{pro:a}
\int_{S_\tau}r^{2-\a}\psi_v^2 d\om dv&\les \ep^2 E_0+ C_1,\\
\label{pro:2} \int_{\tau_1}^{\tau_2}\int_{S_\tau}r^{1-\a}\psi_v^2
d\om dvd\tau&\leq\int_{0}^{\tau_2} \int_{S_\tau}r^{1-\a}\psi_v^2
d\om dvd\tau\les \ep^2 E_0+ C_1.
\end{align}
The second inequality holds for all $\tau_1\leq\tau_2$. We claim
that there exists a dyadic sequence $\{\tau_n\}_{n=3}^{\infty}$ such
that
\begin{equation}
\label{pro:1a} \int_{S_{\tau_n}}r^{1-\a}\psi_v^2dv d\om \leq
(1+\tau_n)^{-1}\left(\ep^2 E_0+ C_1\right),
\end{equation}
where $\tau_n$ satisfies the inequality $\ga^{-2}\tau_n\leq\tau_{n-1}\leq\ga^2\tau_n$ for
 some large constant $\ga$ depending
 on $R$, $\lambda$, $\a$, $h$ and $C_0$. In fact it suffices to show that there exists $\tau_n\in[\ga^n, \ga^{n+1}]$
 such that \eqref{pro:1a} holds. Otherwise
$$\int_{\ga^{k}}^{\ga^{k+1}}\int_{S_\tau}r^{1-\a}\psi_v^2dvd\om d\tau \gtrsim \ln\ga \left(\ep^2 E_0+ C_1\right),$$
which contradicts to ~\eqref{pro:2} if $\ga$ is large enough.

Interpolate between ~\eqref{pro:a} and ~\eqref{pro:1a}. We get
\begin{equation}
\label{pro:1} \int_{S_{\tau_n}}r\psi_v^2dv d\om \les
(1+\tau_n)^{-1+\a}\left(\ep^2 E_0+ C_1\right).
\end{equation}
Now, take $p=1$ in the p-weighted energy inequality
~\eqref{pWEineq}. Using the estimates ~\eqref{pro:1} and conditions
$(b), (c)$, we obtain
\begin{align}
\notag
&\int_{S_{\tau}}r(\pa_v\psi)^2 d\om dv +
\int_{\tau_{n-1}}^{\tau}\int_{S_t}(\pa_v\psi)^2 + |\nabb\psi|^2
d\om dv dt\\
\label{pro:1t}
&\les \int_{S_{\tau_{n-1}}}r(\pa_v\psi)^2 d\om dv +
E[\phi](\tau_{n-1})+ (1+\tau_{n-1})^{-1+\a}C_1\\
\notag
&\les (1+\tau_{n-1})^{-1+\a}\left(\ep^2 E_0+ C_1\right) + E[\phi](\tau_{n-1})\\
\notag
&\les(1+\tau)^{-1+\a}\left(\ep^2 E_0+ C_1\right) + E[\phi](\tau_{n-1})
\end{align}
for all $\tau\in[\tau_{n-1}, \tau_{n}]$.

Our goal is to retrieve the full energy flux through the
hypersurface $\Si_\tau$. Since
\begin{align*}
\int_{S_\tau}(\pa_v\psi)^2 + |\nabb\psi|^2
d\om dv&=\int_{S_\tau}\left(\phi_v^2 + |\nabb\phi|^2\right)r^2
d\om dv+\left.\int_{\om}r\phi^2d\om\right|_{R}^{\infty}\\
&\geq2\int_{S_\tau}J^T_\mu[\phi]n^{\mu}d\si-R\int_{\om}\phi^2d\om
\end{align*}
and
\begin{align}
\label{pro1topro2} R^2\phi^2&=2\int_0^{R}r\phi^2 dr+2\int_0^R
r^2\phi\cdot \pa_r\phi dr\leq2\int_0^{R}r\phi^2
dr+\int_0^{R}r^2\phi^2 dr+\int_0^{R}r^2(\pa_r\phi)^2 dr,
\end{align}
integrate over $[\tau_1, \tau_2]\times S^2$. We obtain
\begin{align*}
 R^2\int_{\tau_1}^{\tau_2}\int_{|\om|=1}\phi^2(\tau, R, \om)d\om
d\tau&\leq\int_{\tau_1}^{\tau_2}\int_{r\leq
R}2\frac{\phi^2}{r}+
\phi^2
+(\pa_r\phi)^2dxd\tau\\
&\les  E[\phi](\tau_1) + (1+\tau_1)^{-2+\a}C_1
\end{align*}
by ~\eqref{ILE} and condition $(c)$. Adding ~\eqref{ILE} and
~\eqref{pro:1t}, we can show that
\begin{align*}
\int_{\tau_{n-1}}^{\tau_n}E[\phi](\tau)d\tau &=\int_{\tau_{n-1}}^{\tau_n}\int_{r\leq R}
\phi_t^2+|\nabla\phi|^2 dxd\tau +
2\int_{\tau_{n-1}}^{\tau_n}\int_{S_\tau}J^T_\mu[\phi]n^\mu
d\si d\tau\\
& \leq\int_{\tau_{n-1}}^{\tau_n}\int_{r\leq R}
\phi_t^2+|\nabla\phi|^2 dxd\tau +
\int_{\tau_{n-1}}^{\tau_n}\int_{S_\tau}\psi_v^2 + |\nabb\psi|^2
dvd\om d\tau\\
&\quad\quad +\int_{\tau_{n-1}}^{\tau_n}\int_{|\om|=1}R\phi^2(\tau,
R, \om)d\om d\tau\\
&\les E[\phi](\tau_{n-1}) + (1+\tau_{n-1})^{-1+\a}\left(\ep^2
E_0+C_1\right).
\end{align*}
Now the energy inequality ~\eqref{eb} shows that for all
$\tau\leq\tau_n$
$$E[\phi](\tau_n)\les E[\phi](\tau)
+D[F]_{\tau}^{\tau_n} \les E[\phi](\tau)+(1+\tau)^{-2+\a}\left(\ep^2
E_0+C_1\right).
$$
Hence we can estimate
$$\int_{\tau_{n-1}}^{\tau_n}\left(E[\phi](\tau_n)-C_1(1+\tau)^{-2+\a}\right)d\tau\les E[\phi](\tau_{n-1})
+ (1+\tau_{n-1})^{-1+\a}\left(\ep^2 E_0+C_1\right), $$ which,
together with the fact that the sequence $\{\tau_n\}$ is dyadic,
implies that
$$E[\phi](\tau_n)\les\tau_n^{-1}E[\phi](\tau_{n-1})+(1+\tau_n)^{-2+\a}\left(\ep^2
E_0+C_1\right).
$$
In particular the energy inequality ~\eqref{eb} shows that
$$E[\phi](\tau)\les \ep^2 E_0+C_1,\quad \forall \tau\geq 0.$$
Therefore
$$E[\phi](\tau_n)\les\tau_n^{-1}E[\phi](\tau_{n-1})+(1+\tau_n)^{-2+\a}\left(\ep^2
E_0+C_1\right)\les \tau_{n}^{-1}\left(\ep^2 E_0+C_1\right),$$ which,
in turn, shows that
\begin{align*}
E[\phi](\tau_n)&\les(1+\tau_n)^{-2+\a}\left(\ep^2 E_0+C_1\right)+\tau_n^{-1}E[\phi](\tau_{n-1})\\
&\les(1+\tau_n)^{-2+\a}\left(\ep^2 E_0+C_1\right)+\tau_n^{-2}\left(\ep^2 E_0+C_1\right)\\
&\les(1+\tau_n)^{-2+\a}\left(\ep^2 E_0+C_1\right).
\end{align*}
The proposition then follows by the fact that $\tau_n$ is dyadic.
\end{proof}
\bigskip

\section{Pointwise Decay of the Solution}
A key ingredient for proving the global existence result for
nonlinear problem is to derive a pointwise decay of the solution.
Merely decay of the energy flux $E[\phi](\tau)$ we have obtained
previously is not sufficient unless we show
 the decay of the energy for the higher derivatives of the solution. In Minkowski space,
 this is a direct consequence of the existence
  of global symmetries. On our inhomogeneous background, no such global symmetry exists.
  To solve this problem, we commute $\Box_g$ with $\Om$, $T$ and control the error terms
   , which are supported in the cylinder $\{(t, x)||x|\leq \frac{1}{2}R\}$, by using elliptic estimates.

\begin{lem}
\label{propofwaveop}
The wave operator $\Box_g$ has the following properties:
\begin{align}
 \label{wavelap}
&g^{ij}\pa_{ij}=\Box_g - g^{00}\pa_{tt}-2g^{0i}\pa_{ti}-\frac{1}{\sqrt{-G}}\pa_\a(g^{\a\b}\sqrt{-G})\pa_\b,\\
\label{commutatorom}
&[\Box_g, \Om]=f^{\a\b}\pa_{\a\b}+f^{\b}\pa_\b,\\
\label{commutatort} &[\Box_g, T]=-\pa_t
g^{\a\b}\cdot\pa_{\a\b}-\pa_t\left(\frac{1}{\sqrt{-G}}\pa_\a(g^{\a\b}\sqrt{-G})\right)\pa_\b,
\end{align}
where $f^{\a\b}$, $f^{\b}$ are smooth functions supported in the cylinder $\{(t, x)||x|\leq \f12 R\}$ and satisfy
\[
 |f^{\a\b}|\les H, \qquad |f^{\b}|\les H.
\]

\end{lem}
\begin{proof}
 We only have to recall the definition of the covariant wave operator 
\[
 \Box_g\phi=\frac{1}{\sqrt{-G}}\pa_{\mu}(g^{\mu\nu}\sqrt{-G}\pa_\nu\phi)
\]
and the fact that the metric $g$ is a perturbation of the Minkowski metric inside the cylinder $\{
r\leq \f12 R\}$.
\end{proof}

\bigskip

The equalities ~\eqref{commutatorom} and ~\eqref{commutatort} show
that in order to derive the energy estimates for $T\phi$ or
$\Om\phi$, we need to estimate the error term $\pa_{\a\b}\phi$.
Since $\Om g$, $T g$ are supported in $\{(t, x)||x|\leq
\frac{1}{2}R\}$, we use elliptic estimates to control the error
terms. For convenience, we may omit the summation sign.

\begin{lem}
\label{errorcontrol} Suppose $g$ satisfies the smallness condition
~\eqref{smallnesscond} or conditions $\mathcal{A}_1$ and
$\mathcal{A}_2$. Then
\begin{align}
\label{hessphibdI} \int_{\tau_1}^{\tau_2}\int_{r\leq \f12
R}|\pa_{\a\b}\phi|^2dxd\tau&\les E[\pa_t\phi](\tau_1)+
E[\phi](\tau_1)+D[\pa_t F]_{\tau_1}^{\tau_2}+D[F]_{\tau_1}^{\tau_2},
\end{align}
\begin{equation}
\label{hessphibd}
 \int_{r\leq R}|\pa_{\a\b}\phi|^2dx\les E[\pa_t\phi](\tau^+)+D[\pa_t F]_{\tau^+}^{\tau+R}+
  E[\phi](\tau^+)+ D[F]_{\tau^+}^{\tau+R},
\end{equation}
where $\tau^+=\max\{\tau-R, 0\}$ and $\pa_{\a\b}\phi$ is the second order derivative in Minkowski space.
\end{lem}
Inequality ~\eqref{hessphibdI} will be used when we commute $\Box_g$ with $\Om$ or $T$ in order to obtain the energy decay for $\Om\phi$ or $T\phi$. It shows that the error terms coming from commutations are under control. Inequality ~\eqref{hessphibd} is useful for deriving the pointwise decay of the solution when $r\leq R$.
\begin{proof}
When $g$ satisfies the smallness condition \eqref{smallnesscond},
the error term can be absorbed because of the smallness assumption.
For the other case, we show that the inequality ~\eqref{hessphibdI}
can be reduced to the condition ~\eqref{morawetz2}.

\bigskip

If $g$ satisfies condition \eqref{smallnesscond}, choose a smooth
cut off function $\chi(x)$, such that
\begin{equation*}
\chi(x)=
\begin{cases}
1, \quad |x|\leq\f12 R,\\
0, \quad |x|\geq R,
\end{cases}
\end{equation*}
and $|\chi|\leq 1$, $|\nabla\chi|\leq\frac{4}{R}$, $|\Delta
\chi|\leq \frac{8}{R^2}$, where we denote $\nabla=(\pa_1,\pa_2,
\pa_3)$. Using the inequality
\[
 (a+b)^2\leq (1+p)a^2+(1+\frac{1}{p})b^2,
\]
we obtain
\begin{align*}
 &\quad\int_{r\leq \f12 R}|\pa_{ij}\phi|^2dx\leq \int_{r\leq R}|\pa_{ij}(\chi\phi)|^2dx=\int_{r\leq R}|\Delta(\chi\phi)|^2dx\\
&\leq \frac{6}{5}\int_{r\leq R}|\Delta\phi|^2dx +C_{R}\int_{r\leq R}|\phi|^2+|\nabla\phi|^2dx\\
&\leq \frac{3}{2}\int_{r\leq R}|\sum g^{ij}\pa_{ij}\phi|^2dx+6\int_{r\leq R}|\sum h^{ij}\pa_{ij}\phi|^2dx+C_R\int_{r\leq R}|\phi|^2+|\nabla\phi|^2dx\\
&\leq \frac{3}{2}\int_{r\leq R}|\sum
g^{ij}\pa_{ij}\phi|^2dx+54H^2\int_{r\leq \f12 R}|\pa_{ij}\phi|^2dx
+C_R\int_{r\leq R}|\phi|^2+|\nabla\phi|^2dx,
\end{align*}
where $C_R$ is a constant depending only on $R$. Recall that $H\leq
\frac{\a}{700(1+\f12 R)^{\a+1}}$. Absorbing the second term, we have
\begin{equation}
 \label{ellipticboundc}
\int_{r\leq \f12 R}|\pa_{ij}\phi|^2dx\leq \frac{8}{5}\int_{r\leq
R}|\sum g^{ij}\pa_{ij}\phi|^2dx + C_R\int_{r\leq
R}|\phi|^2+|\nabla\phi|^2dx.
\end{equation}
Then from equation ~\eqref{wavelap}, we have
\begin{equation*}
 |\sum g^{ij}\pa_{ij}\phi|^2\leq \frac{5}{4}|\pa_{\a t}\phi|^2+
 C|\Box_g\phi|^2+C|\pa\phi|^2.
\end{equation*}
Add $2|\pa_{\a t}\phi|^2$ to both sides of ~\eqref{ellipticboundc}
and then integrate from $\tau_1$ to $\tau_2$. Then the integrated
local energy inequality ~\eqref{ILE} implies that
\begin{equation}
\label{ellipticbdab}
 \int_{\tau_1}^{\tau_2}\int_{r\leq \f12 R}|\pa_{\a\b}\phi|^2dx\leq 4\int_{\tau_1}^{\tau_2}
 \int_{r\leq R}|\pa_{\a t}\phi|^2dx +
 C_R\left(E[\phi](\tau_1)+D[F]_{\tau_1}^{\tau_2}\right).
\end{equation}
We have to bound the first term which can be estimated by the integrated local energy inequality for $\pa_t\phi=T\phi$.
Commute the equation \eqref{THEWAVEEQ} with
the vector field $\pa_t$. Using the identity ~\eqref{commutatort},
we obtain
\begin{align*}
|\Box_g\pa_t\phi|^2&\leq \frac{17}{16}|\sum\pa_t g^{\a\b}\pa_{\a\b}\phi|^2+C |\pa_t\Box_g\phi|^2+C|\pa\phi|^2\\
    &\leq 17 H^2
    |\pa_{\a\b}\phi|^2+C|\pa_t\Box_g\phi|^2+C|\pa\phi|^2
\end{align*}
for $r\leq \frac{1}{2}R$. Denote $C_{g}=18^2\b^2(1+R)^{\a+1}$. By
~\eqref{ILEH2}, we have
\begin{align*}
\int_{\tau_1}^{\tau_2} \int_{r\leq R}|\pa_{\a t}\phi|^2dxd\tau&\leq C_g E[\pa_t\phi](\tau_1)+C_g D[\Box_g\pa_t\phi]_{\tau_1}^{\tau_2}\\
&\leq C_g E[\pa_t\phi](\tau_1) + C D[\pa_t F]_{\tau_1}^{\tau_2}+ C \int_{\tau_1}^{\tau_2} \int_{r\leq \f12 R}|\pa\phi|^2dxd\tau\\
& \quad+C_g 17 H^2(1+\frac{1}{2}R)^{\a+1}\int_{\tau_1}^{\tau_2} \int_{r\leq \f12 R}|\pa_{\a\b}\phi|^2dxd\tau\\
&\leq C_g E[\pa_t\phi](\tau_1) + C D[\pa_t F]_{\tau_1}^{\tau_2}+ C E[\phi](\tau_1)+C D[F]_{\tau_1}^{\tau_2}\\
& \quad+ C_g 17 H^2(1+\f12 R)^{\a+1}\int_{\tau_1}^{\tau_2}
\int_{r\leq \f12 R}|\pa_{\a \b}\phi|^2dxd\tau.
\end{align*}
Plug this into \eqref{ellipticbdab}. Notice that
 $$4C_g 17 H^2(1+\f12 R)^{\a+1}\leq\frac{2^{1+\a}\times68\times 18^2\times (\a\b)^2}{700^2}<1.$$
We thus have shown the estimate ~\eqref{hessphibdI} by absorbing the
first term $\int|\pa_{\a\b}\phi|^2dxd\tau$ on the right hand side.

\bigskip
For the case when $g$ satisfies conditions $\mathcal{A}1$ and
$\mathcal{A}2$, Proposition~\ref{prop2} shows that $(g^{ij})$ is
uniformly elliptic. Thus by elliptic estimates, together with the
conditon ~\eqref{morawetz2}, we obtain
\begin{equation*}
 \begin{split}
  \int_{\tau_1}^{\tau_2}\int_{r\leq \f12 R}|\pa_{\a\b}\phi|^2dxd\tau
&\les \int_{\tau_1}^{\tau_2}\int_{r\leq \f12 R}|\pa_{\a t}\phi|^2dxd\tau +\int_{\tau_1}^{\tau_2}\int_{r\leq R}|\sum g^{ij}\pa_{ij}\phi|^2+\phi^2dxd\tau\\
&\les \int_{\tau_1}^{\tau_2}\int_{r\leq \f12 R}|\pa_{\a t}\phi|^2dxd\tau +E[\phi](\tau_1)+D[F]_{\tau_1}^{\tau_2}\\
&\les  E[\pa_t\phi](\tau_1) +  D[\pa_t F]_{\tau_1}^{\tau_2}+
E[\phi](\tau_1)+ D[F]_{\tau_1}^{\tau_2}.
 \end{split}
\end{equation*}
In any case, we have proven ~\eqref{hessphibdI}.

\bigskip

Now, we use ~\eqref{hessphibdI} to prove ~\eqref{hessphibd}. Since
$(g^{ij})$ is uniformly elliptic as having shown in Proposition
~\ref{prop2}, elliptic estimates together with the estimate
\eqref{wavelap} imply that
\begin{align}
\notag
\int_{r\leq R}|\pa_{\a\b}\phi|^2dx &\les\int_{r\leq R}|\pa_{\a t}\phi|^2dx+\int_{r\leq 2R}|g^{ij}\pa_{ij}\phi|^2 +|\phi|^2dx\\
\label{Htwobd} &\les E[\pa_t\phi](\tau)+E[\phi](\tau)+\int_{r\leq
2R}|F|^2+|\pa_{tt}\phi|^2+|\phi|^2dx,
\end{align}
where we recall that $h^{\a\b}$ are supported in $r\leq \f12 R$. It remains to bound the integral on the larger ball $r\leq 2R$.
We first consider the case when $\tau\geq R$. Take $\tau_1=\tau-R$ and $\tau_2=\tau+R$ in ~\eqref{ILE0}. We have
$$\int_{\tau}^{\tau+R}\int_{r\leq 2R}|\pa_t\phi|^2+\phi^2 dxdt \les \int_{\tau-R}^{\tau+R}\int_{\Si_\tau}
\frac{|\pa\phi|^2+|\frac{\phi}{1+r}|^2}{(1+r)^{1+\a}}d\vol\les
E[\phi](\tau-R)+D[F]_{\tau-R}^{\tau+R}.
$$
Therefore, using Sobolev embedding, we have
\begin{equation}
 \label{phibd}
\left.\int_{r\leq 2R}\phi^2dx\right|_{\tau}\les
E[\phi](\tau-R)+D[F]_{\tau-R}^{\tau+R}.
\end{equation}
Similarly we have
\begin{equation}
 \label{Fbd}
\left.\int_{r\leq 2R}|F|^2dx\right|_{\tau}\les
D[F]_{\tau-R}^{\tau+R}+D[\pa_t F]_{\tau-R}^{\tau+R}.
\end{equation}
We claim that
\begin{equation}
 \label{ttphibd}
\int_{r\leq 2R}|\pa_{tt}\phi|^2dx\les E[\pa_t\phi](\tau-R)+D[\pa_t F]_{\tau-R}^{\tau}+
 E[\phi](\tau-R)+ D[F]_{\tau-R}^{\tau}.
\end{equation}
In fact, consider the region bounded by $\Si_{\tau-R}$ and $t=\tau$.
Take $X=T$ in ~\eqref{energyeq}. We get
\begin{equation}
\label{energyineq2R}
\begin{split}
\int_{r\leq 2R}J^T_\mu[\phi]n^\mu d\si&=\int_{\Si_{\tau-R}\cap
\{t\leq \tau\}}J^T_\mu[\phi]n^\mu d\si-
\int_{\tau-R}^{\tau}\int_{r\leq R +t-\tau}F\pa_t\phi +
K^T[\phi]d\vol.
\end{split}
\end{equation}
Notice that the metric is flat when $r\geq \f12 R$. We can estimate
\[
 \int_{\tau-R}^{\tau}\int_{r\leq R +t-\tau}|K^T[\phi]|d\vol\les\int_{\tau-R}^{\tau}\int_{r\leq R}
 |\pa\phi|^2dxdt\les E[\phi](\tau-R)+D[F]_{\tau-R}^{\tau}.
\]
Using Cauchy-Schwartz inequality, we obtain
\begin{equation*}
 \begin{split}
  \int_{\tau-R}^{\tau}\int_{r\leq R +t-\tau}|F\pa_t\phi|d\vol&\les\int_{\tau-R}^{\tau}\int_{\Si_t}|F|^2(1+r)^{\a+1}+\frac{(\pa_t\phi)^2}{(1+r)^{\a+1}}dxdt\\
&\les E[\phi](\tau-R)+D[F]_{\tau-R}^{\tau}.
 \end{split}
\end{equation*}
In any case, $g$ satisfies the condition $\mathcal{A}_1$. Hence by
Proposition ~\ref{prop2}, we can show that
$$\int_{r\leq 2R}|\pa_t\phi|^2 dx\les\int_{r\leq 2R}J^T_\mu[\phi]n^\mu
d\si
$$
and
\[
 \int_{\Si_{\tau-R}\cap \{t\leq \tau\}}J^T_\mu[\phi]n^\mu d\si\les\int_{\Si_{\tau-R}}J^T_\mu[\phi]n^\mu d\si\les
  E[\phi](\tau-R).
\]
Replace $\phi$ with $\pa_t\phi$ in ~\eqref{energyineq2R}. Then
~\eqref{hessphibdI} implies that
\begin{align*}
&\int_{r\leq 2R}|\pa_{tt}\phi|^2 dx\les E[\pa_t\phi](\tau-R)+D[\Box_g\pa_t\phi]_{\tau-R}^{\tau}\\
&\les E[\pa_t\phi](\tau-R) + D[\pa_t F]_{\tau-R}^{\tau}+\int_{\tau-R}^{\tau}\int_{r\leq \f12 R}|\pa_{\a\b}\phi|^2+|\pa\phi|^2dxdt\\
&\les E[\pa_t\phi](\tau-R)+D[\pa_t F]_{\tau-R}^{\tau}+
E[\phi](\tau-R)+ D[F]_{\tau-R}^{\tau}.
\end{align*}
The inequality ~\eqref{hessphibd} then follows from ~\eqref{Htwobd},
~\eqref{phibd}, ~\eqref{Fbd}, ~\eqref{ttphibd} and ~\eqref{eb}.

When $\tau\leq R$, notice that the initial data are supported in
$r\leq R$. We conclude that the solution $\phi$ is supported in
$\{r\leq \tau +R\}$. And the above inequalities still hold when
replacing $\tau-R$ with 0.
\end{proof}

\begin{remark}
 The proof for the smallness assumption case shows that the condition ~\eqref{morawetz2}
 can be replaced by assuming that the deformation tensor $\pi^{T}_{\a\b}$ is sufficiently small.
 This nevertheless still allows $g$ to be far away from Minkowski metric.
\end{remark}

Having proven the above lemma, we now establish the main proposition
in this section.

\begin{prop}
\label{energydecay}
 Suppose $F$ satisfies the conditions in Proposition \ref{prop4}. Then on $\Si_\tau$, for all $\a<\delta\leq
 1$, we have
  \begin{eqnarray}
 \label{pbout}
\int_{\om}|r\phi|^2d\om\les(1+\tau)^{-1+\delta}\left(\ep^2 E_0 + C_1\right), \quad\quad r\geq R ,\\
\label{pbout1}
\int_{\om}r|\phi|^2d\om\les(1+\tau)^{-2+\a}\left(\ep^2 E_0 +
C_1\right), \quad\quad r\geq  R.
\end{eqnarray}
If in addition,
\begin{description}
\item[(a)] $\pa_t F$ satisfies the same conditions in Proposition
~\ref{prop4};
\item[(b)] $E[\pa_t\phi](\tau)\les\left(\ep^2 E_0 + C_1\right)(1+\tau)^{-2+\a},$
\end{description}
then
\begin{equation}
 \label{pbin}
|\phi|^2\les(1+\tau)^{-2+\a}\left(\ep^2 E_0 + C_1\right), \quad\quad
r\leq R.
\end{equation}

\end{prop}
The estimate ~\eqref{pbin} is stronger than ~\eqref{pbout} and
~\eqref{pbout1} due to the extra condition $(b)$ together with the
robust elliptic estimates on compact region.

\begin{proof} Inequality ~\eqref{pbout1} follows from Lemma ~\ref{lem1} and Proposition ~\ref{prop4}.
We use the p-weighted energy inequalities to prove ~\eqref{pbout}.
Estimate ~\eqref{pro:1t} implies that
\begin{equation}
\label{stauenergy1} \int_{S_\tau}r(\pa_v\psi)^2d\om dv
\les(1+\tau)^{-1+\a}\left(\ep^2 E_0 + C_1\right),\quad\forall
\tau\geq 0.
\end{equation}
Interpolate with ~\eqref{pro:a}. We get
$$\int_{S_\tau}r^{1+\delta-\a}(\pa_v\psi)^2d\om
dv\les(1+\tau)^{-1+\delta}\left(\ep^2 E_0 + C_1\right),\quad\forall
\a<\delta \leq 1.
$$
Hence using ~\eqref{pbout1}, we can estimate
\begin{align*}
\quad\int_{\om}|\psi|^2(\tau,v,\om)d\om
&\les\int_{\om}|\psi|^2(\tau, v_\tau, \om)d\om +\left(\int_{v_\tau}^v\int_{\om}|\pa_v\psi|d\om dv\right)^2\\
                    &\les(1+\tau)^{-2+\a}\left(\ep^2 E_0 + C_1\right) +  \int_{v_\tau}^v\int_{\om}r^{1+\delta-\a}|\pa_v\psi|^2d\om dv\int_{v_\tau}^v r^{-1-\delta+\a}dv\\
                    &\les(1+\tau)^{-1+\delta}\left(\ep^2 E_0 +
                    C_1\right),
\end{align*}
where $v=\frac{r+\tau}{2}$ and $\delta>\a$.

\bigskip

To prove ~\eqref{pbin}, using\eqref{hessphibd}, ~\eqref{phiboundn},
we show that for $r\leq R$
\begin{align*}
 |\phi|^2&\les \int_{r\leq R}|\pa_{ij}\phi|^2+\phi^2dx\\
&\les E[\pa_t\phi](\tau^+)+D[\pa_t F]_{\tau^+}^{\tau+R}+ E[\phi](\tau^+)+ D[F]_{\tau^+}^{\tau+R}\\
&\les(1+\tau^+)^{-2+\a}\left(\ep^2 E_0 + C_1\right)\\
&\les (1+\tau)^{-2+\a}\left(\ep^2 E_0 + C_1\right),
\end{align*}
where $\tau^+=\max\{\tau-R, 0\}$.
\end{proof}

\bigskip

To obtain the poitwise decay of the solution, we need to estimate
$E[\Om\phi](\tau)$, $E[\pa_t\phi](\tau)$. Since $\Box_g$ does not
commute with the vector fields $\Om$, $T$, we use elliptic estimates
to control the errors coming from commutation. We establish a
proposition that gives the decay of the energy
$E[\Om\phi](\tau)$, $E[\pa_t\phi](\tau)$.
\begin{prop}
\label{VectorDecay}
 Let the vector field $X$ be $\pa_t$ or $\Om$. Assume that $F$, $X(F)$
 satisfy conditions $(a), (b)$ and $(c)$(without $E[\phi](\tau_1)$ on the right hand side) in Proposition \ref{prop4}
.
\begin{itemize}
 \item[(1)] If $X=\pa_t$, then
\begin{equation}
\label{Tdecay}
 E[\pa_t\phi](\tau)\les (1+\tau)^{-2+\a}\left(\ep^2 E_0+C_1\right).
\end{equation}
\item[(2)] If $X=\Om$ and $$E[\pa_t\phi](\tau)\les (1+\tau)^{-2+\a}\left(\ep^2 E_0+C_1\right),$$ then
\begin{equation}
\label{Omdecay}
 E[\Om\phi](\tau)\les (1+\tau)^{-2+\a}\left(\ep^2 E_0+C_1\right).
\end{equation}
\end{itemize}
Furthermore, in both cases $\Box_g X(\phi)$ satisfies conditions in
Proposition ~\ref{prop4}, with a new constant $C (C_1+\ep^2 E_0)$,
where $C$ does not depend on $F$ or $\phi$.
\end{prop}
\begin{proof}  Using the identities ~\eqref{commutatorom}, ~\eqref{commutatort}, we can write the equation for $X(\phi)$
\begin{equation*}
 \Box_g X\phi=X(F)+f^{\a\b}\pa_{\a\b}\phi+f^{\b}\pa_\b\phi.
\end{equation*}
Recall that the metric is flat when $r\geq \f12 R$. We infer that
$\Box_g X(\phi)$ satisfies conditions $(a)$ and $(b)$ in Proposition
~\ref{prop4}. We show that $\Box_g X(\phi)$ also satisfies condition
$(c)$. In fact, using the estimate ~\eqref{hessphibdI}, we can show that
\begin{equation}
\label{innerFbdE}
\begin{split}
 D[\Box_g X(\phi)]_{\tau_1}^{\tau_2}& \les D[X(F)]_{\tau_1}^{\tau_2}+\int_{\tau_1}^{\tau_2}\int_{r\leq \f12 R}|\pa_{\a\b}\phi|^2+|\pa\phi|^2dx d\tau\\
&\les D[X(F)]_{\tau_1}^{\tau_2} + D[\pa_t F]_{\tau_1}^{\tau_2}+ D[F]_{\tau_1}^{\tau_2}+ E[\pa_t\phi](\tau_1)+E[\phi](\tau_1)\\
&\les (C_1+\ep^2E_0)(1+\tau_1)^{-2+\a}+E[\pa_t\phi](\tau_1),
\end{split}
\end{equation}
where we have shown $E[\phi](\tau)\les (C_1+\ep^2
E_0)(1+\tau)^{-2+\a}$ by Proposition ~\ref{prop4}. Therefore
\begin{itemize}
 \item[(1)]  when $X=\pa_t$, the above inequality already  implies that $\Box_g X(\phi)$ satisfies the condition $(c)$.
 \item[(2)]  when $X=\Om$, the extra condition on $E[\pa_t\phi](\tau)$ also indicates that $\Box_g X(\phi)$ satisfies the condition $(c)$.
\end{itemize}
Once we have verified that $\Box_g X(\phi)$ satisfies conditions in
Proposition ~\ref{prop4}, we can conclude the estimates
~\eqref{Tdecay}, ~\eqref{Omdecay}.
\end{proof}
\bigskip

This Proposition shows that although $\Box_g$ does not commute with
$\pa_t$, the assumptions on the nonlinearity $F$ are sufficient to
obtain the decay of $E[\pa_t \phi](\tau)$. However, to prove the
decay of $E[\Om\phi](\tau)$, we need to show the decay of
$E[\pa_t\phi]$ first. The idea is that we first commute the equation
with $T$. Then pass the $T$ derivatives to $\Om$ derivatives.
\begin{cor}
\label{VenergydecayCor}
 Suppose $\Om^k T^j F$ satisfies the conditions in Proposition ~\ref{prop4} for all $(k, j)\in A$. Then
\begin{align}
 \label{Venergydecay}
 & E[\Om^k T^j\phi](\tau)\les \left(\ep^2 E_0 +
 C_1\right)(1+\tau)^{-2+\a}.
\end{align}
 \end{cor}
\begin{proof} We prove by induction on $k$. For $k=0$, $j\leq 8$, the estimate \eqref{VenergydecayCor} follows from the
first case in Proposition ~\ref{VectorDecay}. Then the second case
in Proposition ~\ref{VectorDecay} implies that
\eqref{VenergydecayCor} holds for $k=1$, $j\leq7$. Repeat this again
until we have arrived at the case when $k=5$ and $j\leq 3$. That
covers all the cases and the corollary follows.
\end{proof}
\begin{remark}
 In Minkowski, we only have to commute $\Box$ with $T$ for 3 times and with $\Om$ for 5 times.
\end{remark}

\section{Boostrap Argument}
To solve our nonlinear problem, we use the standard Picard iteration process. We prove,
 by a boostrap argument, that the nonlinear term $F$ decays.

\begin{prop}
\label{mainprop}
 Suppose the nonlinearity $F$ satisfies the following conditions
\begin{description}
\item[(a)] $\sum\limits_{(k, j)\in A}\int_{r\leq R}|\Om^k T^jF|^2dx\leq 2E_0\ep^2(1+\tau)^{-3+\a};$
\item[(b)] $\sum\limits_{(k, j)\in B}\int_{r\leq R}|\nabla\Om^kT^jF|^2dx\leq 2E_0\ep^2(1+\tau)^{-3+\a};$
\item[(c)] $\sum\limits_{(k, j)\in A}\int_{\tau_1}^{\tau_2}\int_{S_\tau}|\Om^k T^j F|^2r^{3-\a}d\vol\leq
 2E_0\ep^2(1+\tau_1)^{-2+\a}.$
\end{description}
Then
\begin{align}
\label{nullbdin}
&\sum\limits_{(k, j)\in A}\int_{r\leq R}|\Om^k T^j F|^2dx\les E_0^2\ep^4(1+\tau)^{-3+\a},\\
\label{nullbdinD}
&\sum\limits_{(k, j)\in B}\int_{r\leq R}|\nabla\Om^k T^j F|^2dx\les E_0^2\ep^4(1+\tau)^{-3+\a},\\
\label{nullbdout} &\sum\limits_{(k, j)\in
A}\int_{\tau_1}^{\tau_2}\int_{S_\tau}|\Om^k T^j
F|^2r^{3-\a}d\vol\les E_0^2\ep^4(1+\tau_1)^{-2+\a},
\end{align}
where $\nabla$ denotes the covariant derivative on $\{t=\tau, r\leq R\}$.
\end{prop}

Estimates ~\eqref{nullbdin} and ~\eqref{nullbdout} are sufficient to
conclude our theorem. The extra boostrap assumption $(b)$ is used to
prove ~\eqref{nullbdin}. We use elliptic estimates to prove
~\eqref{nullbdin}, which, in turn, implies ~\eqref{nullbdinD}. For
~\eqref{nullbdout}, we rely on the p-weighted energy inequality and
the null structure of the quadratic nonlinearity of $F$. The
assumptions here, together with Proposition~\ref{prop4} and Proposition \ref{VectorDecay}, imply that
\[
E[\Om^k T^j \phi](\tau)\les E_0\ep^2(1+\tau)^{-2+\a},\quad
\forall(k, j)\in A.
\]

Since the cubic or higher order nonlinearities of $F$ behave better,
it suffices to consider the quadratic nonlinearity
$A^{\mu\nu}\pa_\mu\phi\pa_\nu\phi$ of $F$, where the constants
$A^{\mu\nu}$ satisfy the null condition. For all $(k, j)\in A$, it
is known that
\begin{equation}
\label{OmTN}
 \Om^k T^j F =\Om^k T^j(A^{\a\b}\pa_\a\phi\pa_\b\phi)=\sum A^{\a\b}\pa_\a\Om^{k_1}T^{j_1}\phi\cdot\pa
 _\b\Om^{k_2}T^{j_2}\phi,
\end{equation}
where notice that $[T, \pa_\a]=0$,  $[\Om, \pa_\a]=0$ or $\pa_\b$ up
to a constant and $k_1+k_2\leq k$,  $j_1+j_2\leq j$. For simplicity, in the sequel, we denote $\phi_1=\Om^{k_1}T^{j_1}\phi$,
$\phi_2=\Om^{k_2}T^{j_2}\phi$. Before proving Proposition \ref{mainprop} in details, we prove a lemma exhibiting
 the properties of the sets $A$ and $B$.
\begin{lem}
\label{propofAB}
Assume $(k_1+k_2, j_1+j_2)\in A$. Then
\begin{itemize}
 \item[(1)] $(k_i+2, j_i+1)\in A$ for at least one $i\in\{1, 2\};$
 \item[(2)] $(k_i, j_i)\in B$ for at least one  $i\in\{1, 2\};$
 \item[(3)] If $(k, j)\in A$ or $B$ then $(k', j')\in A$ or $B$ for any $k'\leq k$, $j'\leq j.$
\end{itemize}
 \end{lem}
\begin{proof} For the first property, since $k_1+j_1 + k_2 +j_2\leq 8$, without loss of generality, we assume $k_1 + j_1\leq 4$. If $k_1\leq 3$, then $(k_1 +2, j_1 +1)\in A$ by definition; If $k_1\geq 4$, then $k_2\leq 1$ and $k_2+2+j_2+1\leq 4+3\leq 8$. Thus $(k_2+2, j_2+1)\in A$.

For the second property, without loss of generality, we assume  $j_1\leq j_2$. Then
$k_1+j_1+2\leq k_1+j_1+k_2+j_2\leq 8$, which shows $(k_1, j_1+2)\in A$. By definition, $(k_1, j_1)\in B$. The
 third property holds by the definition.
\end{proof}

\subsection{Proof of ~\eqref{nullbdin} and ~\eqref{nullbdinD}}
Inside the cylinder $\{r\leq R\}$, it is not necessary to require the null structure of the quadratic nonlinearity of $F$. To 
estimate $F$, we rely on the robust elliptic estimates to obtain the pointwise bound of the solution 
and control the rest by the energy inequality.
 
We first prove \eqref{nullbdin}. Notice that 
\begin{equation}
\label{pro3:12}
 \sum\limits_{(k, j)\in A}|\Om^k T^j F|^2  \les
 |\pa\phi_1|^2|\pa\phi_2|^2.
\end{equation}
Since $k_1+k_2\leq k$, $j_1+j_2\leq j$, without loss of generality,
assume $(k_1, j_1)\in B$ according to Lemma ~\ref{propofAB}, that
is, $(k_1, j_1+2)\in A$. We claim that
\begin{equation}
\label{Tptwisebd} |\pa\phi_1|^2\les \ep^2 E_0(1+\tau)^{-2+\a},\quad
r\leq \f12 R.
\end{equation}
In fact, for $\pa_t\phi_1$, since $(k_1, j_1+2)\in A$, Proposition
~\ref{VectorDecay} shows that
 $$E[\pa_{tt}\phi_1]\les E_0\ep^2 (1+\tau)^{-2+\a}$$
  and
$\pa_t\Box_g(\pa_t\phi_1)$ satisfies the conditions in Proposition
~\ref{prop4}. Therefore Proposition ~\ref{energydecay} implies that
\[
 |\pa_t\phi_1|\les \ep^2 E_0(1+\tau)^{-2+\a},\quad r\leq \f12 R.
\]
For $|\nabla\phi_1|^2$, using elliptic estimates and the inequality
\eqref{hessphibd}, we can show that
\begin{align*}
&\int_{r\leq \f12 R}|\nabla\pa_{\a\b}\phi_1|^2dx\les\int_{r\leq R}|g^{ij}\pa_{ij}\nabla\phi_1|^2+|\nabla\pa_{\b t}\phi_1|^2+|\pa_{\a\b}\phi_1|^2+|\pa\phi|^2dx\\
&\les\int_{r\leq R}\left|\left(\Box_g - g^{00}\pa_{tt}-2g^{0i}\pa_{ti}-\frac{1}{\sqrt{-G}}\pa_\a(g^{\a\b}\sqrt{-G})\pa_\b\right)\nabla\phi_1\right|^2 dx\\
&\qquad+\int_{r\leq R}|\nabla\pa_{\b}\pa_t\phi_1|^2+|\pa_{\a\b}\phi_1|^2+|\pa\phi|^2dx\\
&\les\int_{r\leq R}|\nabla\Box_g\phi_1|^2 +|\nabla\pa_{tt}\phi_1|^2 + |\nabla\phi_1|^2dx
+\int_{r\leq \f12 R}|\pa_{\a\b}\phi_1|^2+|\pa\phi_1|^2dx\\
&\qquad +\int_{r\leq R}|\nabla\pa_{\b}\pa_t\phi_1|^2+|\pa_{\a\b}\phi_1|^2dx+|\pa\phi|^2\\
&\les \int_{r\leq R}|\nabla\Box_g\phi_1|^2 dx + \ep^2
E_0(1+\tau^+)^{-2+\a},
\end{align*}
where $\tau^{+}=\max\{\tau-R, 0\}$. Now can write
\begin{align*}
 \nabla\Box_g\phi_1=\nabla \Box_g\Om^{k_1}T^{j_1}\phi
&=\nabla\Om^{k_1}T^{j_1}F + \sum\limits_{\substack{k\leq k_1,j\leq j_1\\k+j<k_1+j_1}} f^{\a\b}_{kj}\nabla\pa_{\a\b}\Om^{k}T^{j}\phi+ f^{\a}_{kj}\nabla\pa_\a\Om^{k}T^{j}\phi\\
&\qquad+\sum\limits_{k\leq k_1,j\leq j_1}
F^{\a\b}_{kj}\pa_{\a\b}\Om^{k}T^{j}\phi+
F^{\a}_{kj}\pa_\a\Om^{k}T^{j}\phi,
\end{align*}
where $f^{\a\b}_{kj}, f^{\a}_{kj}$, $F^{\a\b}_{kj}, F^{\a}_{kj}$ are smooth functions depending on $g$
 and are supported in $r\leq \f12 R$. Therefore, using the boostrap assumptions $(b)$ and the estimate ~\eqref{hessphibd},
  we obtain
\begin{align*}
 &\int_{r\leq \f12 R}|\nabla\pa_{\a\b}\phi_1|^2dx\les E_0\ep^2(1+\tau)^{-2+\a}+\sum\limits_{1'<1}
 \int_{r\leq \f12 R}|\nabla\pa_{\a\b}\phi_{1'}|^2dx,
\end{align*}
where $1'<1$ means $k_{1'}\leq k_1$, $j_{1'}\leq j_1$ and
$k_{1'}+j_{1'}<k_1+j_1$. By induction, we get
\[
\int_{r\leq \f12 R}|\nabla\pa_{\a\b}\phi_1|^2dx\les
E_0\ep^2(1+\tau)^{-2+\a}.
\]
Then Sobolev embedding and the estimate ~\eqref{hessphibd} imply
that
\begin{equation}
\label{Caest}
\begin{split}
 \|\nabla\phi_1\|_{C^\f12(B_{\f12 R})}^2&\les \int_{r\leq \f12 R}|\pa_{ij}\nabla\phi_1|^2+|\nabla\phi_1|^2dx\\
&\les \int_{r\leq \f12 R}|\nabla\pa_{ij}\phi_1|^2+|\pa_{\a\b}\phi_1|^2+|\pa\phi_1|^2dx\\
&\les E_0\ep^2(1+\tau)^{-2+\a}.
\end{split}
\end{equation}
We thus have shown the desired estimate \eqref{Tptwisebd}, which
implies that
\begin{align*}
\sum\limits_{(k, j)\in A}\int_{r\leq \f12 R}|\Om^k T^j F|^2dx &\les(1+\tau)^{-2+\a}E_0 \ep^2 \int_{r\leq\f12
R}|\pa\phi_2|^2dx\\
&\les(1+\tau)^{-2+\a}E_0 \ep^2 E[\phi_2](\tau)\\
&\les(1+\tau)^{-3+\a}E_0^2 \ep^4.
\end{align*}
When $\f12 R\leq r\leq R$, we use the angular momentum $\Om$.
Observe that $k_1+k_2\leq 5$. Using Sobolev embedding on the unit
sphere, we always have
\begin{equation}
 \label{SemSphere}
\int_{\om}|\pa\phi_1|^2\cdot|\pa\phi_2|^2d\om\les\int_{\om}|\pa\phi_{1'}|^2d\om\cdot\int_{\om}|\pa\phi_{2'}|^2d\om,
\end{equation}
where $k_{1'}\leq k_1+2$, $k_{2'}=k_2$ if $k_1\leq k_2$; otherwise,
$k_{1'}=k_1$, $k_{2'}\leq k_2+2$. In any case, $(k_{1'}, j_{1}),
(k_{2'}, j_2)\in A$. Since $j_1+j_2=j$, without loss of generality,
we assume $(k_{1'}, j_1+1)\in A$. Therefore, we have
\begin{align}
\label{SemrR}
&\int_{\om}|\pa\phi_{1'}|^2d\om\les \int_{r\leq R}|\pa\phi_{1'}|^2+|\nabla\pa\phi_{1'}|^2dx\\
\notag
&\les E[\phi_{1'}](\tau)+E[\pa_t\phi_{1'}](\tau^+)+D[\pa_t F]_{\tau^{+}}^{\tau+R}+E[\phi_{1'}](\tau^+)+D[F]_{\tau^{+}}^{\tau+R}\\
\notag &\les (1+\tau)^{-2+\a}E_0 \ep^2,
\end{align}
where $\tau^+=\max\{\tau-R, 0\}$. Then ~\eqref{SemSphere} implies
that
\begin{align*}
\sum\limits_{(k, j)\in A}\int_{\f12 R\leq r\leq R}|\Om^k T^j F|^2dx &\les(1+\tau)^{-2+\a}E_0 \ep^2 \int_{r\leq R}|\pa\phi_{2'}|^2dx\\
&\les(1+\tau)^{-2+\a}E_0 \ep^2 E[\phi_{2'}](\tau)\\
&\les(1+\tau)^{-3+\a}E_0^2 \ep^4.
\end{align*}
This completes the proof for ~\eqref{nullbdin}.
\begin{remark}
 We remark here that ~\eqref{SemrR} is only true when $r$ is bigger than a constant. That is why we need
 to distinguish the two cases: $r\leq \f12 R$ and $\f12 R\leq r\leq
 R$.
\end{remark}

Now we use \eqref{nullbdin} to prove ~\eqref{nullbdinD}. By
~\eqref{OmTN} and ~\eqref{SemSphere}, for all $(k, j)\in B$, we can
show that
\begin{equation*}
 \int_{\om}|\nabla\Om^k
T^j
F|^2d\om\les\int_{\om}|\pa\phi_1|^2\cdot|\pa_{\b\gamma}\phi_2|^2d\om\les\int_{\om}|\pa\phi_{1'}|^2d
\om\cdot\int_{\om}|\pa_{\b\gamma}\phi_{2'}|^2d\om,
\end{equation*}
where $k_{1'}\leq k_1+2$, $k_{2'}=k_2$ if $k_1\leq k_2$; otherwise,
$k_{1'}=k_1$, $k_{2'}\leq k_2+2$.
 In any case, $(k_{1'}, j_{1}), (k_{2'}, j_{2})\in B$. Thus by ~\eqref{Tptwisebd} and
 ~\eqref{SemrR}, we have
$$\int_{\om}|\pa\phi_{1'}|^2d\om\les (1+\tau)^{-2+\a}\ep^2 E_0,\quad r\leq
R.
$$
On the other hand, ~\eqref{hessphibd} implies that
$$
\int_{r\leq R}|\pa_{\b\gamma}\phi_{2'}|^2dx\les
(1+\tau)^{-2+\a}\ep^2 E_0.
$$
Therefore
\begin{align*}
 \sum\limits_{(k, j)\in B}\int_{r\leq R}|\nabla\Om^k T^j F|^2dx&\les(1+\tau)^{-2+\a}\ep^2
 E_0\int_{r\leq R}|\pa_{\b\gamma}\phi_{2'}|^2dx\les
 E_0^2\ep^4(1+\tau)^{-3+\a}.
\end{align*}
 We thus have shown ~\eqref{nullbdinD}.
\begin{remark}
 It is not necessary to require the nonlinearity to satisfy the null condition when $r\leq R$. Moreover, when $r\leq R$, the nonlinearity can be any form of
\[
 F=F(\pa\phi, \phi), \quad F(0, 0)=0,\quad DF(0, 0)=0.
\]
\end{remark}

\subsection{Proof of ~\eqref{nullbdout}}
For this part, we rely on the p-weighted energy inequality and the
null structure of the nonlinearity. For the null form
$N=A^{\mu\nu}\pa_\mu\phi\pa_\nu\phi$, we have
$$\Om^k T^j N(\pa\phi, \pa\phi)=\sum N(\pa\phi_1, \pa\phi_2),$$
where $\phi_1=\Om^{k_1}T^{j_1}\phi$, $\phi_2=\Om^{k_2}T^{j_2}\phi$ and $k_1+k_2\leq k, j_1+j_2\leq j$. 
The p-weighted energy inequality is an estimate in terms of $\psi=r\phi$ instead of $\phi$. For this reason, we expand $F$ in $\psi$.
\begin{lem}
If $N$ is a null form, then
\begin{equation}
\label{nullform} r^4|\Om^k T^j N|^2\les\sum\limits_{1,
2}\phi_1^2\phi_2^2+\phi_1 ^2\cdot r^2\pa_r^2\phi_2+|\nabb\psi_1|^2|
\nabb\psi_2|^2+|\pa_v\psi_1|^2|\pa_u\psi_2|^2,
\end{equation}
where $v=\frac{t+r}{2}, u=\frac{t-r}{2}$.
\end{lem}
\begin{proof} In fact, notice that
$$r^2 N(\pa\phi_1, \pa\phi_2)=\phi_1\phi_2+r(\phi_1\phi_2)_r+N(\pa\psi_1, \pa\psi_2)
$$
and
$$|N(\pa\psi_1, \pa\psi_2)|\les |\pa_v\psi_1|\cdot|\pa_u\psi_2|+|\pa_u\psi_1|\cdot|\pa_v\psi_2|+|\nabb\psi_1
|\cdot|\nabb\psi_2|.
$$
The lemma then follows.
\end{proof}

\bigskip

Using this lemma, it suffices to consider the four terms on the
right hand side of ~\eqref{nullform}. We handle the first three
terms in a uniform way. Let $\Phi_1$ be $\phi_1$ or $\nabb\psi_1$;
$\Phi_2$ be $\phi_2$, $r\pa_r\phi_2$ and $\nabb\psi_2$
correspondingly. Since $k_1+k_2\leq 5$, using Sobolev embedding on
the unit sphere as we did in ~\eqref{SemSphere}, we obtain
\begin{equation}
\label{pro3om}
\int_{\om}|\Phi_1|^2|\Phi_2|^2d\om\les\int_{\om}|\Phi_{1'}|^2d\om\cdot\int_{\om}|\Phi_{2'}|^2d\om,
\end{equation}
where $k_{1'}\leq k_1+2$, $k_{2'}=k_2$ if $k_1\leq k_2$; otherwise,
$k_{1'}=k_1$, $k_{2'}\leq k_2+2$. Here we omit the summation sign.
For the third case when $\Phi_1=\nabb\psi_1$, $\Phi_2=\nabb\psi_2$,
without loss of generality, we assume $k_1\leq k_2$. Thus by the
fact that $\nabb\psi_1=\Om\phi_1$, we can always write $\Phi_{1'}$
as $\Om^k T^j\phi$ for some $(k, j)\in A$. Then ~\eqref{pbout} shows
that
$$
r^2\int_{\om}|\Phi_{1'}|^2d\om \les(1+\tau)^{-1+\delta}\ep^2
E_0,\quad\forall \delta>\a.
$$
Therefore, by the integrated energy inequality ~\eqref{ILE0}, we
have
\begin{equation}
\label{phiphi}
\begin{split}
\int_{\tau_1}^{\tau_2}\int_{S_\tau}r^{-1-\a}\Phi_{1}^2\Phi_{2}^2 d\vol&=\int_{\tau_1}^{\tau_2}\int_{v_\tau}^{\infty}\int_{\om}r^{1-\a}\Phi_{1}^2\Phi_{2}^2 dvd\om d\tau\\
&\les \int_{\tau_1}^{\tau_2}\int_{v_\tau}^{\infty} r^{1-\a}
\int_{\om}|\Phi_{1'}|^2d\om\int_{\om}|\Phi_{2'}|^2d\om dvd\tau\\
&\les(1+\tau_1)^{-1+\delta}\ep^2 E_0\int_{\tau_1}^{\tau_2}\int_{S_\tau}\frac{|\Phi_{2'}|^2}{r^{3+\a}}d\vol\\
&\les(1+\tau_1)^{-1+\delta}\ep^2 E_0 (1+\tau_1)^{-2+\a}\ep^2 E_0\\
&\les (1+\tau_1)^{-2+\a}\ep^4 E_0^2.
\end{split}
\end{equation}

\bigskip

It remains to estimate the main terms
$|\pa_v\psi_1|^2|\pa_u\psi_2|^2$ in ~\eqref{nullform}. There are two
cases according to which one is bigger: $k_1$ or $k_2$.

\bigskip

\textbf{When} $k_1\leq k_2$, in particular we have $k_1\leq 2$. The
idea is that we bound $|\pa_v\psi_1|$ uniformly and then control
$|\pa_u\psi_2|^2$ by the energy flux through the null hypersurface
$v=constant$. We first establish a lemma to show that the energy
flux through $v=constant$ is bounded.
\begin{lem}
\label{crossnullen}
 Consider the region $D=[u_1, u_2]\times[v_1, \infty)\subset S_\tau \times [\tau_1, \tau_2]$. Then
$$\int_{u_1}^{u_2}\int_{|\om|=1}(\pa_u\psi_2)^2d\om du\les (1+\tau_1)^{-2+\a}\ep^2
E_0.
$$
\end{lem}
\begin{proof} Back to the energy equation ~\eqref{energyeq}, take $X=T$ on the region $D$. We get
\begin{align*}
\int_{u_1}^{u_2}J_\mu^T[\phi_2]n^{\mu}d\si + \int\limits_{\substack{v\geq v_1\\u=u_1 }}J_\mu^T[\phi_2]n^{\mu}d\si&=\int\limits_{\substack{v\geq v_1\\u=u_2 }}J_\mu^T[\phi_2]n^{\mu}d\si+\int_{I_{\tau_1}^{\tau_2}}J_\mu^T[\phi_2]n^{\mu}d\si\\
&\quad+ \int_D \Box_g\phi_2\cdot \pa_t\phi_2d\vol.
\end{align*}
Thus applying Cauchy-Schwartz inequality to the last term, we obtain
\begin{align*}
\int_{u_1}^{u_2}\int_{\om}r^2(\pa_u\phi_2)^2d\om du &\leq 2\int_{u_1}^{u_2}J_\mu^T[\phi_2]n^{\mu}d\si\\
&\les E[\phi_2](\tau_1)+D[\Box_g\phi_2]_{\tau_1}^{\tau_2}+\int_D \frac{(\pa_t\phi_2)^2}{(1+r)^{1+\a}}d\vol \\
&\les(1+\tau_1)^{-2+\a}\ep^2 E_0.
\end{align*}
Then ~\eqref{pbout1} implies that
\begin{align*}
 \int_{u_1}^{u_2}\int_{|\om|=1}(\pa_u\psi_2)^2d\om du &= \int_{u_1}^{u_2}\int_{|\om|=1}r^2(\pa_u\phi_2)^2 + \pa_u(r\phi_2^2)d\om du\\
&=\int_{u_1}^{u_2}\int_{|\om|=1}r^2(\pa_u\phi_2)^2 d\om du + \left.\int_{\om}r\phi_2^2d\om\right|_{u_1}^{u_2}\\
&\les (1+\tau_1)^{-2+\a}\ep^2 E_0 +(1+\tau_2)^{-2+\a}\ep^2 E_0\\
&\les (1+\tau_1)^{-2+\a}\ep^2 E_0.
\end{align*}
\end{proof}

We continue our proof when $k_1\leq 2$. The above lemma shows that
\begin{equation}
\label{pro3sup}
\begin{split}
\int_{\tau_1}^{\tau_2}\int_{S_\tau}r^{-1-\a}|\pa_v\psi_1|^2|\pa_u\psi_2|^2 d\vol
&=\int_{v_{\tau_1}}^{\infty}\int_{u_{\tau_1}}^{u(v)}|\pa_u\psi_2|^2r^{1-\a}|\pa_v\psi_1|^2d ud\om dv\\
 &\leq \int_{v_{\tau_1}}^{\infty}\int_{u_{\tau_1}}^{u(v)}|\pa_u\psi_2|^2 d\om du\cdot \sup\limits_{u, \om}r^{1-\a}|\pa_v\psi_1|^2 dv\\
& \les (1+\tau_1)^{-2+\a}\ep^2
E_0\int_{v_{\tau_1}}^{\infty}\sup\limits_{u,
\om}r^{1-\a}|\pa_v\psi_1|^2 dv.
\end{split}
\end{equation}
We use Sobolev embedding to estimate $\sup\limits_{u,
\om}r^{1-\a}|\pa_v\psi_1|^2$. First, on the unit sphere, we have
\begin{align}
\label{supremev}
 |\pa_v\psi_1|^2 \les\sum\limits_{a\leq 2}\int_{\om}|\Om^a\pa_v\psi_1|^2d\om&\les\sum
 \limits_{a\leq
 2}\int_{\om}|\pa_v\Om^a\psi_1|^2\om=\int_{\om}|\pa_v\psi_{1'}|^2d\om,
\end{align}
where $k_{1'}\leq k_1 +2$ and we omit the summation sign. Then on
$[u_1, u_2]$, we have
\begin{align*}
r^{1-\a}(\pa_v\psi_{1'})^2&\les\left.r^{1-\a}(\pa_v\psi_{1'})^2\right|_{u=u_1} + \int_{u_1}^{u_2}r^{1-\a}(\pa_v\psi_{1'})^2du \\
&\quad\quad+ \int_{u_1}^{u_2}r^{1-\a}(\pa_u\pa_v\psi_{1'})^2du + \int_{u_1}^{u_2}r^{-\a}(\pa_v\psi_{1'})^2du\\
&\les \left.r^{1-\a}(\pa_v\psi_{1'})^2\right|_{u=u_1} + \int_{u_1}^{u_2}r^{1-\a}(\pa_v\psi_{1'})^2du\\
& \quad\quad+ \int_{u_1}^{u_2}r^{1-\a}(\lap\psi_{1'})^2+
r^{3-\a}|F_{1'}|^2du,
\end{align*}
where we use the wave equation ~\eqref{waveqpsi} and
$u_1=u_\tau=\frac{\tau-R}{2}, u_2=u(v)$. Integrate over the unit
sphere. We can show that
\begin{align*}
 &\int_{v_{\tau_1}}^{\infty}\int_{\om}r^{1-\a}|\pa_v\psi_{1'}|^2 d\om dv \les \int_{S_{\tau_{1}}}r^{1-\a}(\pa_v\psi_{1'})^2dvd\om\\
& +\int_{\tau_{1}}^{\tau_2}\int_{S_\tau}r^{1-\a}(\pa_v\psi_{1'})^2 +
(\nabb\Om\phi_{1'})^2r^{1-\a} +r ^{3-\a}|F_{1'}|^2 dvd\om d\tau,
\end{align*}
where $\nabb=\frac{\Om}{r}$. We claim that the above inequality can
be bounded by a multiple of $\ep^2 E_0$. In fact, the first term can
be bounded by $(1+\tau_1)^{-1+\a}\ep^2 E_0$ by ~\eqref{stauenergy1};
the second term can be bounded by $\ep^2 E_0$  by ~\eqref{pro:2};
the third term can be controlled by $(1+\tau_1)^{-2+\a}\ep^2 E_0$ by
the integrated energy inequality ~\eqref{ILE0}(notice that
$k_{1'}\leq k_1+2\leq 4$); The last term is good by our bootstrap
assumptions and Proposition ~\ref{VectorDecay}. Summarizing, we have
shown
$$\int_{v_{\tau_1}}^{\infty}\sup\limits_{u, \om}r^{1-\a}|\pa_v\psi_1|^2 dv\les \ep^2
E_0.
$$
Plug this into ~\eqref{pro3sup}. We get
$$\int_{\tau_1}^{\tau_2}r^{1-\a}\int_{S_\tau}|\pa_v\psi_1|^2|\pa_u\psi_2|^
2 dvd\om d\tau\les (1+\tau_1)^{-2+\a}\ep^4 E_0^2.$$

\bigskip

\textbf{When} $k_2\leq k_1$, that is $k_2\leq 2$. For this case, we
control $|\pa_v\psi_1|^2$ by the p-weighted energy inequality and
bound $\pa_u\psi_2$ uniformly. First using Sobolev embedding on the
unit sphere, we have
\begin{align}
\label{supremeu}
 |\pa_u\psi_2|^2 \les\sum\limits_{a\leq 2}\int_{\om}|\Om^a\pa_u\
 \psi_2|^2d\om&\les\sum\limits_{a\leq
 2}\int_{\om}|\pa_u\Om^a\psi_1|^2d\om=\int_{\om}|\pa_u\psi_{2'}|^2d\om,
\end{align}
where $k_{2'}\leq k_2 +2$ and we omit the summation sign. Therefore,
by ~\eqref{pro:2}, we obtain
\begin{equation}
 \label{pro3sup1}
\begin{split}
 &\int_{\tau_1}^{\tau_2}\int_{S_\tau}r^{1-\a}|\pa_v\psi_1|^2|\pa_u\psi_2|^2 dvd\om d\tau\\
&\les \int_{\tau_1}^{\tau_2}\int_{S_\tau}r^{1-\a}|\pa_v\psi_1|^2 \cdot\int_{\om}|\pa_u\psi_{2'}|^2d\om \quad dvd\om d\tau\\
&\les (1+\tau_1)^{-1+\a}\ep^2 E_0\int_{\tau_1}^{\tau_2}
\sup\limits_{v}r^{-\a}\int_{\om}|\pa_u\psi_{2'}|^2d\om d\tau.
\end{split}
\end{equation}
Then for all $v\in[v_\tau,\infty)$, we can show that
\begin{align*}
 r^{-\a}(\pa_u\psi_{2'})^2&\les \left.r^{-\a}(\pa_u\psi_{2'})^2\right|_{v=v_{\tau_2}}+\left|\int_{v}^{v_{\tau_2}}r^{-1-\a}|\pa_u\psi_{2'}|^2dv\right|\\
&\quad\quad +2\left|\int_{v}^{v_{\tau_2}}r^{-\a}|\pa_u\psi_{2'}\cdot\pa_v\pa_u\psi_{2'}|dv\right|\\
&\les\left.r^{-\a}(\pa_u\psi_{2'})^2\right|_{v=v_{\tau_2}}+\int_{v_\tau}^{\infty}r^{-1-\a}|\pa_u\psi_{2'}|^2dv\\
&\quad\quad + \int_{v_\tau}^{\infty}r^{-1-\a}(\pa_u\psi_{2'})^2dv + \int_{v_\tau}^{\infty}r^{1-\a}(\pa_v\pa_u\psi_{2'})^2dv\\
&\les\left.r^{-\a}(\pa_u\psi_{2'})^2\right|_{v=v_{\tau_2}}+ \int_{v_\tau}^{\infty}\frac{(\pa_u\psi_{2'})^2}{r^{1+\a}}dv\\
&\quad\quad +\int_{v_\tau}^{\infty}r^{1-\a}(\lap\psi_{2'})^2dv+
\int_{v_\tau}^{\infty}r^{3-\a}|F_{2'}|^2dv.
\end{align*}
Integrating over the unit sphere, we conclude that
\begin{align}
\notag
&\int_{\tau_1}^{\tau_2}  \sup\limits_{v}r^{-\a}\int_{\om}|\pa_u\psi_{2'}
|^2d\om d\tau\les \int_{\tau_1}^{\tau_2}\int_{\om}\left.r^{-\a}(\pa_u\psi_{2'})^2\right|_{v=v_{\tau_2}}d\tau \\
\label{dupsi}
&\quad\quad+\int_{\tau_1}^{\tau_2}\int_{S_\tau}\frac{(\pa_u\psi_{2'})^2}{r^{1+\a}}+r^{1-\a}(\nabb\Om\phi_{2'})^2
+r^{3-\a}|F_{2'}|^2dvd\om d\tau.
\end{align}
Reparametrize the first term. We get
$$\int_{\tau_1}^{\tau_2}\int_{\om}\left.r^{-\a}(\pa_u\psi_{2'})^2\right|_{v=v_{\tau_2}}d\tau=\int_{u_{\tau_1}
}^{u_{\tau_2}}\int_{\om}r^{-\a}|\pa_u\psi_{2'}|^2d\om
du\leq\int_{u_{\tau_1}}^{u_{\tau_2}}\int_{\om}|\pa_u
\psi_{2'}|^2d\om du,
$$
which can be bounded by $(1+\tau_1)^{-2+\a}\ep^2 E_0$ by Lemma
~\ref{crossnullen}. The second and third term in ~\eqref{dupsi} can
also be bounded by $(1+\tau_1)^{-2+\a}\ep^2 E_0$ because of
~\eqref{ILE0}(notice that $k_{2'}\leq k_2+2\leq 4$). The last term
in \eqref{dupsi} is good by our assumptions and Proposition
~\eqref{VectorDecay}. Therefore
\[
 \int_{\tau_1}^{\tau_2}\int_{S_\tau}r^{1-\a}|\pa_v\psi_1|^2|\pa_u\psi_2|^2
 dvd\om d\tau\les (1+\tau_1)^{-1+\a}\ep^2 E_0\cdot(1+\tau_1)^{-2+\a}\ep^2 E_0\les(1+\tau_1)^{-2+\a}\ep^4
 E_0^2.
\]
This concludes the proof of ~\eqref{nullbdout} and hence Proposition
~\ref{mainprop} follows.

\section{Proof of the Main Theorem }
We used the foliation $\Si_\tau$, part of which is null, in the
previous argument. However, we do not have a local existence result
with respect to the foliation $\Si_\tau$. We thus use the standard
Picard iteration process. Take $\phi_{-1}(t, x)=0$. We solve the
following linear wave equation recursively
\begin{equation}
\label{iteration}
\begin{cases}
 \Box_{g(t, x)}\phi_{n+1}=F(\phi_n, \pa\phi_n), \\
 \phi_{n+1}(0,x)=\ep \phi_0(x), \pa_t\phi_{n+1}(0,x)=\ep \phi_1(x).
\end{cases}
\end{equation}
Now suppose the implicit constant in Proposition ~\ref{mainprop} is
$C_1$, which, according to our notation, depends only on $R$, $\a$,
$\lambda$, $h$ and $C_0$. Set
$$\ep_0=\frac{1}{\sqrt{C_1 E_0}}.
$$
Then for all $\ep\leq \ep_0$, we have
\[
 C_1\ep^4 E_0^2\leq \ep^2 E_0.
\]
Thus by the continuity of $F(\phi_n, \pa\phi_n)$, we in fact have
shown that the nonlinear term $F$ satisfies
\begin{align*}
 &\sum\limits_{(k, j)\in A}\int_{r\leq R}|\Om^k T^j F(\phi_n, \pa\phi_n)|^2 dx\leq C_1 E_0^2\ep^4 (1+\tau)^{-3+
 \a}\leq E_0\ep^2 (1+\tau)^{-3+\a}, \quad\forall n,\\
 &\sum\limits_{(k, j)\in A}\int_{\tau_1}^{\tau_2}\int_{S_\tau}|\Om^k T^j F(\phi_n, \pa\phi_n)|^2r^{3-\a}d\vol\leq E_0\ep^2
 (1+\tau)^{-2+\a},\quad\forall n.
\end{align*}
Then Corollary ~\ref{VenergydecayCor} implies that
\[
 E[\Om^k T^j\phi_{n}](\tau)\les E_0 \ep^2 (1+\tau)^{-2+\a},\quad\forall (k, j)\in
 A.
\]
We remark here that all the implicit constants are independent of
$n$.

Then Proposition ~\ref{energydecay}, together with the Sobolev
embedding on the unit sphere, implies that
\begin{align*}
 &\sum\limits_{k\leq 2, j\leq 2}|\Om^k T^j\phi_n|\les \sqrt{E_0}\ep (1+r)^{-\f12}(1+|t-r+R|)^{-1+\frac{1}{2}\a},\\
&|\phi_n|\les_\delta\sqrt{E_0}\ep
(1+r)^{-1}(1+|t-r+R|)^{-\f12+\frac{1}{2}\delta},\quad \forall
\delta>\a.
\end{align*}
We still need to show that the solution is $C^2$. The first step is to show that $\phi_n$ is uniformly bounded in $C^1$.
When $r\leq \f12 R$,
the inequality \eqref{Tptwisebd} implies that
\begin{equation*}
 |\pa\Om^k T^j\phi_n|\les \sqrt{E_0}\ep
 (1+r)^{-\f12}(1+\tau)^{-1+\frac{1}{2}\a},\quad \forall k\leq 1, j\leq 2.
\end{equation*}
When $\f12 R\leq r\leq R$, we use \eqref{SemrR} and Sobolev embedding on the unit sphere and we can obtain the same estimates.
When $r\geq R$, noticing that $\nabb=\frac{\Om}{r}$, it suffices to consider $\pa_r \phi_n$. First the inequality \eqref{dupsi} shows
that
$$\int_{\tau_1}^{\tau_2}  \sup\limits_{v}r^{-\a}\int_{\om}|\pa_u\psi|^2+|\pa_t\pa_u\psi|^2d\om d\tau
\les(1+\tau_1)^{-2+\a}\ep^2E_0,
$$
where $\psi=r\Om^{k}T^{j}\phi_n$, $k\leq 3, j\leq 2$. Using Sobolev
embedding on $S^2\times[\tau_1, \tau_2]$, we obtain
$$|r\pa_u\Om^k T^j\phi_n|^2\les \phi^2+
r^{\a}(1+\tau)^{-2+\a}\ep^2E_0,\quad\forall k\leq 1, j\leq 2.
$$
Since $\pa_u=\pa_t-\pa_r$ and $|\Om^k T^j\phi_n|^2, |\pa_t\Om^k
T^j\phi_n|^2\les (1+r)^{-1}(1+\tau)^{-2+\a}\ep^2E_0$ and $\a<1$, we
obtain
$$|\pa_r\Om^k T^j\phi_n|\les (1+r)^{-\f12}(1+\tau)^{-1+\f12\a}\sqrt{E_0}\ep,\quad \forall k\leq 1, j\leq 2.$$
In particular, we have shown that 
\[
 |\pa \Om^k T^j \phi_n|\les (1+r)^{-\f12}(1+\tau)^{-1+\f12\a}\sqrt{E_0}\ep,\quad \forall k\leq 1, j\leq 2.
\]
To show that $\phi_n$ is also bounded in $C^2$. Outside the cylinder $\{r\leq \frac{1}{4}R\}$, we use the equation. 
Back to the equation ~\eqref{iteration}, we
can represent $\pa_{rr}\phi_n$ as follows
$$g^{rr}\pa_{rr}\phi_n=F(\phi_{n-1},\pa\phi_{n-1})-g^{\a\b}\pa_{\a\b}\phi_n+g^{rr}\pa_{rr}\phi_
n-\frac{1}{\sqrt{-G}}\pa_{\a}\left(g^{\a\b}\sqrt{-G}\right)\pa_\b\phi_n.
$$
Since we have shown that
\begin{align*}
 &|\pa\phi_n|, |\Om^2\phi_n|, |\pa_{tt}\phi_n|, |\phi_n|, |\pa \Om\phi_n|, |\pa T\phi_n|\les
 (1+r)^{-\f12}(1+\tau_1)^{-1+\f12\a}\sqrt{E_0}\ep, \quad\forall n,
\end{align*}
we thus can estimate
\[
 |\pa^2\phi_n|\les \sum\limits_{k, j\leq 1}|\pa \Om^k T^j\phi_n|+|\pa_{rr}\phi_n|\les\sqrt{E_0}\ep
 (1+r)^{-\f12}(1+\tau)^{-1+\f12\a}.
\]
Inside the cylinder $\{r\leq \frac{1}{4}R\}$, we rely on elliptic theory. First, we have the elliptic equation
 for $\phi_{n+1}$
\[
 g^{ij}\pa_{ij}\phi_{n+1}=F(\phi_n, \pa \phi_n)-\frac{1}{\sqrt{-G}}\pa_{\a}\left(g^{\a\b}\sqrt{-G}\right)\pa_\b\phi_{n+1}-g^{00}
T^2\phi_{n+1}-2g^{0i}\pa_i T\phi_{n+1}.
\]
As we have shown from estimates \eqref{Caest} that
\begin{equation*}
 \|\pa\Om^k T^j\phi_n\|_{C^{\f12}(B_{\f12 R})}\leq \|\nabla\Om^k T^j\phi_n\|_{C^{\f12}(B_{\f12 R})}+
\|\Om^k T^{j+1}\phi_n\|_{C^0}\les \sqrt{E_0}\ep
 (1+r)^{-\f12}(1+\tau)^{-1+\frac{1}{2}\a}
\end{equation*}
for all $k\leq 1, j\leq 2$,
we conclude that the right hand side of the above elliptic equation is uniformly bounded in $C^{\f12}(B_{\f12 R})$. Hence
Schauder estimates \cite{elliptic} imply that
\[
 \|\phi_{n+1}\|_{C^{2, \f12}(B_{\frac{1}{4}R})}\les \sqrt{E_0}\ep
 (1+r)^{-\f12}(1+\tau)^{-1+\frac{1}{2}\a}.
\]
In particular, the above argument shows that
\begin{equation}
\label{ptdcphin}
 \sum\limits_{|\b|\leq 2}|\pa^{\b}\phi_n|\les \sqrt{E_0}\ep
 (1+r)^{-\f12}(1+\tau)^{-1+\f12\a}.
\end{equation}
Now the classical local theory shows that there exists a time $t^*>0$
and a unique smooth solution $\phi(t, x)\in C^\infty([0, t^*)\times
\mathbb{R}^{3})$ of the equation ~\eqref{THEWAVEEQ}. Moreover, the
proof of the local theory indicates that
\[
\phi_{n}(t, x)\rightarrow \phi(t, x)
\]
in $C^{\infty}([0, t^*)\times \mathbb{R}^3)$ since the proof of the local
existence result relies on the Picard iteration. Therefore by
~\eqref{ptdcphin}, we have pointwise bound for the solution $\phi$
\[
\sum\limits_{|\b|\leq 2}|\pa^{\b}\phi|\les \sqrt{E_0}\ep
(1+r)^{-\f12}(1+\tau)^{-1+\f12\a},\quad\forall (t, x)\in [0,
T)\times \mathbb{R}^3.
\]
By a theorem of H$\ddot{o}$rmander ~\cite{hormander} that as long as
the solution is bounded up to the second order derivatives, the
solution exists globally. That is there exists a unique global
solution $\phi(t,x)\in C^\infty(\mathbb{R}^{3+1})$ which solves
~\eqref{THEWAVEEQ}. Moreover
\[
\phi_{n}(t, x)\rightarrow \phi(t, x),\quad (t, x)\in
\mathbb{R}^{3+1}.
\]
Therefore $\phi$ admits all the estimates of $\phi_n$ obtained above
.

\section{Acknowledgments}
The author is greatly indebted to his advisor Igor Rodnianski for suggesting and leading to the problem. He thanks Igor Rodnianski for sharing numerous valuable thoughts and insights. He also thanks Jonathan Luk and Pin Yu for helpful discussions and comments on the manuscript.

\bibliography{shiwu}{}
\bibliographystyle{plain}

\bigskip

Department of Mathematics, Princeton University, NJ 08544 USA

\textsl{Email}: shiwuy@math.princeton.edu

\end{document}